\documentclass[oneside,reqno,12pt]{amsart}

\usepackage{lutz-mafen}

\begin{document}
\title{Matroids arising from electrical networks}
\author{Bob Lutz}
\address{Comcast Technology Center, 1800 Arch St., Philadelphia, PA 19103}
\email{bob\_lutz@comcast.com}
\thanks{Work of the author was partially supported by NSF grants DMS-1401224 and DMS-1701576; and by NSF grant DMS-1440140 while the author was in residence at the Mathematical Sciences Research Institute in Berkeley, California during the fall 2019 semester.}
\date{\today}

\begin{abstract}
This paper introduces \emph{Dirichlet matroids}, a generalization of graphic matroids arising from electrical networks. We present four main theorems. First, we exhibit a matroid quotient involving geometric duals of networks embedded in surfaces with boundary. Second, we characterize the Bergman fans of Dirichlet matroids as subfans of graphic Bergman fans. Third, we prove an interlacing result on the real zeros and poles of the trace of the response matrix. And fourth, we bound the coefficients of the precoloring polynomial of a network by the coefficients of the associated chromatic polynomial.
\end{abstract}

\maketitle

\section{Introduction}

An \emph{electrical network} (or simply a \emph{network}) $D=(G,\B)$ consists of a finite connected graph $G=(V,E)$ and a nonempty set $\B\subseteq V$ called the \emph{boundary}. This paper associates a matroid $M(D)$ to any network $D$ as follows.

Let $\eo=E\cup\eh$, where $\eh$ is an extra element not in $E$. A \emph{grove} of $D$ is a nonempty acyclic set $F\subseteq E$ such that $F$ meets every interior vertex, and every component of $F$ meets the boundary. Let $M(D)$ denote the matroid on $\eo$ whose bases are all sets $B\subseteq \eo$ of the following types:
\begin{enumerate}[(i)]
	\item $B=F$ for some grove $F$ of $D$ containing
		exactly one path between boundary vertices
	\item $B=F\cup\eh$ for some grove $F$ of $D$
		containing no path between boundary vertices.
\end{enumerate}
A \emph{Dirichlet matroid} is any matroid of the form $M(D)$.

For example, the \emph{star network} $D=S_d$ is the network on $d$ vertices illustrated on the left side of Figure \ref{fig:introstar}. Here we have a single interior vertex $i$ and an edge $ij$ for every $j\in \B$. The bases of $M(D)$ are the 2-element subsets of $\eo$. Hence the Dirichlet matroid $M(D)$ is isomorphic to the uniform matroid $U_{2,d}$, i.e. the $d$-point line.

\begin{figure}[ht]
	\centering
	\begin{tikzpicture}[scale=1.3]
	\coordinate (a) at (0,0);
	\coordinate (b1) at (210:1);
	\coordinate (b2) at (150:1);
	\coordinate (b3) at (90:1);
	\coordinate (b4) at (30:1);
	\coordinate (b5) at (-30:1);
	\draw (a)--(b1);
	\draw (a)--(b2);
	\draw (a)--(b3);
	\draw (a)--(b4);
	\draw (a)--(b5);
	\draw[fill=black] (270:0.5) circle (0.4pt);
	\draw[fill=black] (250:0.5) circle (0.4pt);
	\draw[fill=black] (290:0.5) circle (0.4pt);
	\draw[fill=black] (a) circle (2.5pt);
	\draw[fill=white] (b1) circle (2.5pt);
	\draw[fill=white] (b2) circle (2.5pt);
	\draw[fill=white] (b3) circle (2.5pt);
	\draw[fill=white] (b4) circle (2.5pt);
	\draw[fill=white] (b5) circle (2.5pt);
	
	\def\a{2.5}
	\coordinate (c) at (\a,0.25);
	\coordinate (d) at ({\a+1},0.25);
	\coordinate (e) at ({\a+2},0.25);
	\coordinate (f1) at ({\a+2.45},0.25);
	\coordinate (f11) at ({\a+2.46},0.32);
	\coordinate (f12) at ({\a+2.44},0.273);
	\coordinate (f13) at ({\a+2.46},0.227);
	\coordinate (f14) at ({\a+2.44},0.18);
	\coordinate (f2) at ({\a+2.95},0.25);
	\coordinate (f21) at ({\a+2.46+0.5},0.32);
	\coordinate (f22) at ({\a+2.44+0.5},0.273);
	\coordinate (f23) at ({\a+2.46+0.5},0.227);
	\coordinate (f24) at ({\a+2.44+0.5},0.18);
	\coordinate (g) at ({\a+3.4},0.25);
	\coordinate (h) at ({\a+4.4},0.25);
	\draw (c) -- (f1);
	\draw (f2) -- (h);
	\draw[fill=black] ({\a+2.58},0.25) circle (0.4pt);
	\draw[fill=black] ({\a+2.7},0.25) circle (0.4pt);
	\draw[fill=black] ({\a+2.82},0.25) circle (0.4pt);
	\draw plot [smooth, tension = 1] coordinates {(f11) (f12) (f13) (f14)};
	\draw plot [smooth, tension = 1] coordinates {(f21) (f22) (f23) (f24)};
	\draw[fill=black] (c) circle (2.5pt);
	\draw[fill=black] (d) circle (2.5pt);
	\draw[fill=black] (e) circle (2.5pt);
	\draw[fill=black] (g) circle (2.5pt);
	\draw[fill=black] (h) circle (2.5pt);
\end{tikzpicture}
	\caption{A star network with boundary vertices marked in white, left;
		an affine diagram of its Dirichlet matroid, right.}
	\label{fig:introstar}
\end{figure}

Dirichlet matroids arise in several ways. Classically, they are the complete principal truncations of graphic matroids along edge sets of cliques (Appendix \ref{sec:appa}). In the language of biased graphs, they are the complete lift matroids of almost-balanced biased graphs (Appendix \ref{sec:printrun}). Our interest is motivated by \emph{Dirichlet arrangements}, a class of hyperplane arrangements introduced in \cite{lutz2019hyp}. The simple Dirichlet matroids are precisely the matroids of Dirichlet arrangements (Example \ref{eg:dirarr}).

Our first theorem generalizes a result on geometric duals of cellularly embedded graphs in surfaces. Let $\Sigma$ be a surface with boundary. An \emph{embedding} of $D$ in $\Sigma$ is an embedding of $\g$ in $\Sigma$ that satisfies certain conditions on the faces and boundary vertices (Definition \ref{def:embedding}). For example, a network is \emph{circular} or \emph{cylindrical} if it is embedded in a closed disk or cylinder, respectively. If $D$ is embedded in $\Sigma$, then one can construct a geometric dual network $D^*$ also embedded in $\Sigma$ (Definition \ref{def:dual}).

Let $M^*(D)$ denote the dual of $M(D)$. If $|\B|\geq 2$, then the ranks of $M^*(D)$ and $M(D^*)$ differ by $|\B|-\chi(\Sigma)-1$, so the two matroids are not isomorphic in general. However, they are closely related. Given matroids $M$ and $N$ on the same ground set, we say that $N$ is a \emph{quotient} of $M$ if every flat of $N$ is a flat of $M$. Matroid quotients are combinatorial abstractions of quotient spaces in linear algebra.

\begin{thm}
Let $D$ be a network embedded in a surface with boundary. If $|\B|\geq 2$, then $M(D^*)$ is a quotient of $M^*(D)$.
\label{thm:dualintro}
\end{thm}

When $|\B|=2$, we recover the special case mentioned above: if $\g$ is cellularly embedded in a surface, then the graphic matroid $M(\g^*)$ is a quotient of $M^*(\g)$  \cite[Lemma 1]{richter1984}. While circular and cylindrical networks have been studied previously \cite{kenyon2017, lam2012}, the author is unaware of any prior results on networks embedded in general surfaces with boundary.

Our second theorem is a tropical characterization of Dirichlet matroids. The \emph{Bergman fan} of a loopless matroid $M$, denoted by $\berg{M}$, is a polyhedral fan supported on the intersection of the tropical hyperplanes defined by the circuits of $M$ (Definition \ref{def:berg}). Just as $M$ is defined by its bases or circuits, it is also defined by $\berg{M}$. For our purposes, an \emph{isomorphism} of polyhedral fans is a piecewise-linear homeomorphism between their supports that maps cones onto cones.

\begin{thm}
Let $\mg$ denote the graph obtained from $\g$ by adding an edge between each pair of boundary vertices. If $\g$ is loopless, then there is an isomorphism of polyhedral fans
\begin{equation}
	\berg{M(D)}\to \berg{M(\mg)}_{\B},
\end{equation}
where $\berg{M(\mg)}_{\B}$ is the subfan of $\berg{M(\mg)}$ consisting of all cones whose points are constant on the set of edges between boudary vertices.
\label{thm:bergintro}
\end{thm}

For a complete graph $K_d$, the Bergman fan $\berg{M(K_d)}$ can be identified with the space of phylogenetic $d$-trees \cite{ardila2006}. We generalize this result to a family of ``complete'' networks and phylogenetic trees with prescribed sets of equidistant leaves (Proposition \ref{prop:newphylo}). We also compute a minimal tropical basis of $M(D)$ (Proposition \ref{prop:mintrop}).

Our third result concerns the \emph{response matrix} of $D$, which describes the electrical properties of the network. For this, we think of the edges of $\g$ as linear resistors and the boundary vertices as nodes to which voltages are applied, e.g. by batteries. Given an edge $e\in E$, the \emph{conductance} $\con_e\in \R$ measures how easily current may flow through $e$. There is a linear map $\R^{\B}\to \R^{\B}$ that takes the boundary voltages to the resulting currents. The \emph{response matrix} is the matrix $\Lambda=\Lambda(D,\gamma)$ of this map. The entries of $\Lambda$ are homogeneous rational functions in the conductances that can be written in terms of groves.

We are interested in the trace of $\Lambda$. A line in $\R^k$ has \emph{positive direction vector} if it can be parametrized as $t\mapsto x+ty$ for some $x\in \R^k$ and $y\in (0,\infty)^k$.

\begin{thm}
The zeros and poles of the trace of the response matrix $\Lambda$ interlace as the conductances of the network travel along any line in $\R^E$ with positive direction vector.
\label{thm:realintro}
\end{thm}

For networks with $|\B|=2$, the quantity $\tr\Lambda$ has an electrical interpretation: it is twice the \emph{effective conductance} between the boundary vertices \cite[Section 2]{wagner2005}. We show that Theorem \ref{thm:realintro} follows from the \emph{half-plane property} of Dirichlet matroids (Definition \ref{def:hpp}).

Our fourth result concerns the polynomial at the heart of the following coloring problem:
\begin{quote}
How many ways can a given injective $k$-coloring of $\B$ be extended to a proper $k$-coloring of $\g$?
\end{quote}
The answer is $\cp{D}(k)$, where $\cp{D}$ is the \emph{precoloring polynomial} of $D$ \cite[Definition 3.5]{lutz2019hyp}. Precoloring polynomials have been studied in connection with Sudoku puzzles, marked order polytopes and list colorings, but little is known about their behavior in general \cite{herzberg2007, jochemko2014, stanley2015}. Our result compares the polynomial $\cp{D}$ directly to the chromatic polynomial $\cp{\g}$.

\begin{thm}
Assume that $\g$ is loopless, and write the precoloring polynomial of $D$ and the chromatic polynomial of $\g$ as
\begin{equation}
	\begin{aligned}
		\cp{\g}(\l) &= \l^d - a_1\l^{d-1} + \cdots + (-1)^d a_d\\
		\cp{D}(\l) &= \l^n - b_1\l^{n-1} + \cdots + (-1)^n b_n,
	\end{aligned}
\end{equation}
so that all $a_i$ and $b_i$ are nonnegative. We have $a_i\geq b_i$ for $i=1,\ldots,n$, with $a_i = b_i$ if $i$ is less than the minimum number of edges in a path between boundary vertices.
\label{thm:cpintro}
\end{thm}
\section{Basics of Dirichlet matroids}
\label{sec:bias}

In this section, we give cryptomorphic definitions of Dirichlet matroids and provide several examples. We also discuss the connection to hyperplane arrangements and characterize the fields over which a given Dirichlet matroid is representable. We assume some familiarity with matroid theory, as presented in \cite{oxley2006}.

\subsection{Cryptomorphic definitions and key examples}
\label{sec:crypt}

Throughout, let $D=(\g,\B)$, where $\g=(V,E)$ is a finite connected graph and $\B\subseteq V$ is a nonempty set. The elements of $\B$ are called the \emph{boundary nodes} of $D$. A \emph{cycle} of $\g$ is a set $C\subseteq E$ such that each vertex of $\g$ is an endpoint of exactly 0 or 2 edges in $C$. Let $d=|V|$, $m=|\B|$, and $n=d-m$.

Recall that $M(D)$ is a matroid on $\eo=E\cup \eh$, where $\eh$ is an element not in $E$, and $\eh$ is shorthand for $\{\eh\}$. For basic descriptions of Dirichlet matroids, we appeal to the theory of \emph{biased graphs} developed by Zaslavsky \cite{zaslavsky1989, zaslavsky1991, zaslavsky2003}. Biased graphs encode a large class of matroids, including Dirichlet matroids. We review this connection in detail in Appendix \ref{sec:appa}.

\begin{mydef}
A \emph{crossing} of $D$ is a minimal (with respect to inclusion) set $X\subseteq E$ containing a path between distinct boundary nodes.
\end{mydef}

\begin{prop}[Circuits]
A set $C\subseteq \eo$ is a circuit of $M(D)$ if and only if one of the following holds:
\begin{enumerate}[(A)]
	\item $C = X\cup \eh$ for some crossing $X$
	\item $C\subseteq E$ is a cycle or loop of $\g$ meeting at
		most 1 boundary node
	\item $C\subseteq E$ is a minimal set containing 2 distinct
		crossings and no circuits of type (B).
\end{enumerate}
	
\begin{proof}
This is proven in Appendix \ref{sec:bias}.
\end{proof}
\label{prop:circuits}
\end{prop}

As we saw in the introduction, the bases of $M(D)$ depend on \emph{groves}, a generalization of spanning trees from the literature on electrical networks \cite{kenyon2011}.

\begin{mydef}
A nonempty acyclic set $F\subseteq E$ is a \emph{grove} of $D$ if $F$ meets every vertex interior vertex, and every connected component of $F$ meets the boundary.
\end{mydef}

\begin{mydef}
Let $\sunc$ be the set of all groves $F$ of $D$ that contain no crossing of $D$. Let $\scro$ be the set of all groves $F$ of $D$ that contain exactly 1 crossing of $D$. 
\label{def:suncscro}
\end{mydef}

\begin{prop}[Bases]
A set $B\subseteq \eo$ is a basis of $M(D)$ if and only if one of the following holds:
\begin{enumerate}[(A)]
	\item $B\in \scro$
	\item $B=F\cup \eh$ for some $F\in \sunc$.
\end{enumerate}
	
\begin{proof}
This is proven in Appendix \ref{sec:bias}.
\end{proof}
\label{prop:bases}
\end{prop}

Sometimes we will prefer to deal only with the simple Dirichlet matroids. We say that a network $D$ is \emph{simple} if $\g$ is simple and $\B$ induces an edgeless subgraph.

\begin{prop}[Simplicity]
A network $D$ is simple if and only if $M(D)$ is simple.
	
\begin{proof}
Let $D$ be a network. The loops of $M(D)$ are the loops of $\g$. The parallel classes of $M(D)$ are the sets of parallel edges of $\g$ and the set $S\cup \eh$, where $S$ is the set of edges between boundary nodes. The result follows.
\end{proof}
\label{prop:simp}
\end{prop}

\begin{eg}[Graphic matroids]
If $|\B|=2$, then $M(D)$ is graphic. In particular, we have $M(D)\cong M(\mg)$, where $\mg$ is the graph obtained from $\g$ by adding an edge between the two boundary nodes. This example is generalized in Proposition \ref{prop:grapheq}.
\label{eg:graphic}
\end{eg}

\begin{eg}[Lines]
Recall the star network $D=S_d$ from the introduction. The bases of $M(D)$ are the 2-element subsets of $\eo$. Hence $D$ is isomorphic to the $d$-point line $U_{2,d}$.
	
\label{eg:star}
\end{eg}

\begin{eg}[Discriminantal arrangements]
Given a point $z\in \C^m$ with all $z_i$ distinct, let $\AA_{m,n}(z)$ denote the arrangement of hyperplanes in $\C^n$ of the following forms:
\begin{equation*}
	\begin{cases}
		x_i=x_j & \mbox{for all } 1 \leq i,j \leq n\\
		x_i=z_j & \mbox{for all } 1 \leq i \leq n \mbox{ and }1\leq j\leq m,
	\end{cases}
\end{equation*}
where the $x_i$ are the coordinates of $\C^n$. The arrangements $\AA_{m,n}(z)$ are called \emph{discriminantal arrangements}. Discriminantal arrangements have been studied extensively for their connections to mathematical physics \cite{cohen2000, manin1989, varchenko2011}. We associate to $\AA_{m,n}(z)$ the matroid defined by its cone (see \cite{stanley2007}).
	
Let $D_{m,n}$ denote the network in which $\g$ is obtained from a complete graph $K_{m+n}$ by deleting all edges between the $m$ boundary nodes. The networks $D_{m,n}$ play the role of complete graphs, in the sense that every simple network is obtained from one of the form $D_{m,n}$ by deletion. The matroid $M(D_{m,n})$ is isomorphic to the matroid of $\AA_{m,n}(z)$.
\label{eg:dmn}
\end{eg}

\begin{eg}[Dirichlet arrangements]
Generalizing the previous example, let $D$ be any simple network, and let $u\in \C^{\B}$ with all $u_i$ distinct. The vector $u$ is called the \emph{Dirichlet boundary data}; we think of $u$ as voltages applied to the boundary nodes of $D$. Let $\AA_D(u)$ denote the arrangement of hyperplanes in $\C^{V\setminus \B}$ given by
\begin{equation*}
	\begin{cases}
		x_i=x_j & \mbox{for all } ij \in E \mbox{ not meeting } \B\\
		x_i=u_j & \mbox{for all } ij \in E \mbox{ with } j \in \B.
	\end{cases}
\end{equation*}
The arrangements $\AA_D(u)$ are called \emph{Dirichlet arrangements}. Dirichlet arrangements model a certain electrical problem as a geometric problem on their complements \cite{lutz2019hyp}. They are also called \emph{$\psi$-graphical arrangements} in connection with order polytopes of finite posets, and \emph{graphic discriminantal arrangements} in connection with the topology of graphic arrangements \cite{stanley2015, cohen2021}. The matroids defined by the cones of these arrangements are the simple Dirichlet matroids.
\label{eg:dirarr}
\end{eg}

\subsection{Matrix representations}
\label{sec:matrep}
We characterize the fields over which $M(D)$ is representable and determine which Dirichlet matroids are graphic.

\begin{mydef}
Let $i\in V\setminus \B$. The \emph{block} of $D$ containing $i$ is the set $U\subseteq V$ of all vertices $j$ such that there exists a path $i_1\cdots i_k$ in $\g$ with $i_1 = i$, $i_k = j$, and $i_1,\ldots,i_{k-1}\in V\setminus \B$.
\label{def:block}
\end{mydef}

In other words, the block containing $i\in V\setminus \B$ is the set of all vertices reachable from $i$ by paths that do not breach the boundary.

\begin{eg}
Consider the network $D$ illustrated in Figure \ref{fig:blocks} with boundary nodes marked in white. For each $i\in V\setminus \B$, let $U(i)$ denote the block of $D$ containing $i$. There are exactly 2 blocks of $D$: $U(i_2)=\{i_1,i_2,i_3\}$ and $U(i_4)=U(i_5)=\{i_3,i_4,i_5,i_6\}$.
	
\begin{figure}[ht]
	\centering
	\begin{tikzpicture}[scale=1.5]
	\coordinate (a) at (0,0);
	\coordinate (b) at (1,0);
	\coordinate (c) at (2,0);
	\coordinate (d) at (3,0);
	\coordinate (e) at (4,0);
	\coordinate (f) at (5,0);
	\draw (a) node[below=1mm] {$i_1$} -- 
		(b) node[below=1mm] {$i_2$} -- 
		(c) node[below=1mm] {$i_3$} -- 
		(d) node[below=1mm] {$i_4$} -- 
		(e) node[below=1mm] {$i_5$} -- 
		(f) node[below=1mm] {$i_6$};
	\draw[fill=white] (a) circle (2pt);
	\draw[fill=black] (b) circle (2pt);
	\draw[fill=white] (c) circle (2pt);
	\draw[fill=black] (d) circle (2pt);
	\draw[fill=black] (e) circle (2pt);
	\draw[fill=white] (f) circle (2pt);
\end{tikzpicture}
	\caption{A network with two blocks.}
	\label{fig:blocks}
\end{figure}
\end{eg}

\begin{mydef}
Let $\omega(D)$ be the positive integer given by
\begin{equation*}
	\omega(D)=\max|U\cap \B|,
\end{equation*}
where the maximum runs over all blocks $U$ of $D$.
\end{mydef}

\begin{prop}
The matroid $M(D)$ is representable over $\KK$ if and only if $|\KK|\geq \omega(D)$.
\label{prop:representable}
	
\begin{proof}
See Appendix \ref{sec:reppf}.
\end{proof}
\end{prop}

\begin{prop}
The following are equivalent:
\begin{enumerate}[(a)]
	\item $\omega(D)=2$
	\item $M(D)$ is binary
	\item $M(D)$ is regular
	\item $M(D)$ is graphic.
\end{enumerate}

\begin{proof}
Equivalence of (a), (b) and (c) follows from Proposition \ref{prop:representable}. We prove equivalence of (a) and (d). Graphic matroids are regular. Hence if $\omega(D)\geq 3$, then Proposition \ref{prop:representable} implies that $M(D)$ is not graphic.
	
Suppose instead that $\omega(D)=2$. We construct a graph $H$ such that $M(D)\cong M(H)$. Start with $H=K_2$ on the vertex set $\{i,j\}$. Given $S\subseteq V$, let $\g(S)$ denote the subgraph of $\g$ induced by $S$. For every block $U$ of $D$ containing only 1 boundary node, take a copy of $\g(U)$ and attach it to $H$ by gluing the boundary node to either $i$ or $j$. For every block $U$ containing 2 boundary nodes, take a copy of $\g(U)$ and attach it to $H$ by gluing one boundary node to $i$ and the other to $j$. Finally, for every edge of $\g$ between boundary nodes, add an edge to $H$ parallel to $ij$. The resulting graph $H$ satisfies $M(D)\cong M(H)$. An explicit isomorphism is given by swapping $\eh$ and $ij$.
	\end{proof}
\label{prop:grapheq}
\end{prop}
\section{Matroid quotients and dual networks}
\label{sec:dual}

Let $M$ and $N$ be matroids on $E$. If every flat of $N$ is a flat of $M$, then we say that $N$ is a \emph{quotient} of $M$, or alternatively that there is a \emph{matroid quotient} $M\to N$. This terminology is explained by the following observation. Given a vector space $V$ and a function $\varphi: E\to V$, consider the matroid $M(\varphi)$ on $E$ in which a set is independent if and only if the images of its elements under $\varphi$ are linearly independent. If $U$ is a linear subspace of $V$ and $f:V\to V/U$ is the quotient map, then $M(f\circ\varphi)$ is a quotient of $M(\varphi)$ \cite[Proposition 7.4.8.2]{brylawski1986}.

\begin{prop}[{\cite[Proposition 8.1.6]{kung1986}}]
Let $M$ and $N$ be matroids on $E$. The following are equivalent:
\begin{enumerate}[(i)]
	\item $N$ is a quotient of $M$
	\item Every circuit of $M$ is a union of circuits of $N$
	\item For any sets $A\subseteq B\subseteq E$ we have
		\begin{equation*}
			\rk_M(B)-\rk_M(A)\geq \rk_N(B)-\rk_N(A).
		\end{equation*}
\end{enumerate}
\label{defthm:quot}
\end{prop}

There are many other cryptomorphic definitions of matroid quotients, but we will not need them here. We turn our attention to embeddings of networks in surfaces with boundary.

\begin{mydef}[Network embedding]
Let $D = (\g,\B)$ be a network and $\Sigma$ a surface with boundary  $\partial\Sigma$. An \emph{embedding} of $D$ in $\Sigma$ is an embedding of $\g$ in $\Sigma$ that satisfies the following conditions:
\begin{enumerate}
	\item $\B\subseteq \partial\Sigma$
	\item The closure of each face of $\g$ is homeomorphic to a closed disk
	\item No face of $\g$ meets more than one component of $\partial\Sigma$.
\end{enumerate}
\label{def:embedding}
\end{mydef}

The conditions of Definition \ref{def:embedding} allow us to define a geometric dual of $D$ as follows. An example is illustrated in Figure \ref{fig:dual}.
 
\begin{mydef}[Geometric dual]
Suppose that $D$ is embedded in $\Sigma$. We construct another network $D^*$ also embedded in $\Sigma$. Draw one vertex of $D^*$ in each face of $\g$; if a face meets $\partial \Sigma$, draw that vertex in $\partial \Sigma$. The boundary nodes of $D^*$ are the vertices in $\partial \Sigma$. Two vertices of $D^*$ share an edge if the corresponding faces of $\g$ share an edge of $\g$. This defines a unique network $D^*$, called the \emph{dual} of $D$.
\label{def:dual}
\end{mydef}

\begin{figure}[ht]
	\centering
	\begin{tikzpicture}[scale=1.8]
\draw[opacity=0.4] (0,0) circle (1);
\foreach \s in {1,2,3,4,5}
{
	\coordinate (a\s) at ({-162+72*\s}:1);
};

\coordinate (g) at (0.3,-0.65);
\coordinate (c) at (0.1,0.1);
\coordinate (f) at (0.6,-0.1);
\coordinate (e) at (-0.3,-0.5);
\coordinate (b) at (-0.6,0.4);
\coordinate (a) at (0.25,0.6);
\coordinate (d) at (0.67,0.3);
\draw (a) -- (a3) -- (d) -- (f) -- (a2) -- (g) -- (a1) -- (e) -- (a5) -- (b) -- (a4) -- (a);
\draw (b) -- (c) -- (a) -- (d);
\draw (a5) -- (c);
\draw (e) -- (c) -- (f);
\draw (e) -- (g);

\foreach \s in {a,b,c,d,e,f,g}
{
	\draw[fill=black] (\s) circle (1.125pt);	
};

\foreach \s in {1,2,3,4,5}
{
	\draw[fill=black] (a\s) circle (1.125pt);
};

\foreach \s in {1,2,3,4,5}
{
	\coordinate (b\s) at ({90+72*\s}:1);
};

\coordinate (h) at ($0.5*(a)+0.5*(b) + (-0.3,0.1)$);
\coordinate (i) at ($0.33*(a3) + 0.33*(a) + 0.33*(d)$);
\coordinate (j) at ($0.33*(a5) + 0.33*(b) + 0.33*(c)$);
\coordinate (k) at ($0.33*(a5) + 0.33*(c) + 0.33*(e)$);
\coordinate (l) at ($0.33*(c) + 0.33*(d) + 0.33*(f)$);
\coordinate (m) at ($0.33*(f) + 0.33*(g) + 0.33*(a2) - (0.1,0.05)$);
\coordinate (n) at ($0.33*(a1) + 0.33*(e) + 0.33*(g)$);
\draw[dashed] (b2) -- (k) -- (j) -- (b1) -- (h) -- (b5) -- (i) -- (b4) -- (m) -- (b3) -- (n) -- (b2);
\draw[dashed] (j) -- (h) -- (l) -- (b4);
\draw[dashed] (i) -- (l);
\draw[dashed] (k) -- (m) -- (l);
\draw[dashed] (m) -- (n);
\foreach \s in {h,i,j,k,l,m,n}
{
	\draw[fill=white] (\s) circle (1.125pt);	
};
\foreach \s in {1,2,3,4,5}
{
	\draw[fill=white] (b\s) circle (1.125pt);
};
\end{tikzpicture}
	\caption{A circular network and its geometric dual, with one vertex set marked in white and the other in black.}
	\label{fig:dual}
\end{figure}

\begin{mydef}
A \emph{bond} of $D$ is a minimal set $B\subseteq E$ such that $E\setminus B$ contains no crossings.
\label{def:insul}
\end{mydef}

\begin{prop}[Cocircuits]
Let $\og$ be the graph obtained from $\g$ by contracting $\B$ to a single vertex. A set $X\subseteq \oe$ is a cocircuit of $M(D)$ if and only if one of the following holds:
\begin{enumerate}[(i)]
	\item $X$ is a cocircuit of $M(\og)$
	\item $X=B\cup \eh$ for some bond $B$ of $D$.
\end{enumerate}
	
\begin{proof}
See Appendix \ref{sec:bias}.
\end{proof}
\label{prop:cocircuits}
\end{prop}

\subsection{Proof of Theorem \ref{thm:dualintro}}

Consider the quotient space $\overline{\Sigma}$ obtained from $\Sigma$ by identifying all points in $\partial\Sigma$ as a single point. We adapt the proof of \cite[Theorem 4.3]{em2015} to argue the following lemma.

\begin{lem}
Let $\Sigma$ be a surface with boundary. If $\g$ is embedded in $\overline{\Sigma}$ with a vertex at the singular point, then $M(\g^*)$ is a quotient of $M^*(\g)$.

\begin{proof}
By Proposition \ref{defthm:quot} it suffices to show that if $A\subseteq B\subseteq E$, then
\begin{equation}
	\rk_{M^*(\g)}(B) - \rk_{M^*(\g)}(A) 
		\geq \rk_{M(\g^*)}(B)-\rk_{M(\g^*)}(A).
	\label{eq:rk0}
\end{equation}
Let $c_G(A)$ denote the number of components of the spanning subgraph $(V,A)$ of $G$. In this notation we have $\rk_{M(G)}(A)=|V|-c_G(A)$. Hence $\rk_{M^*(G)}(A)=|A|-c_G(E\setminus A) + 1$, so for any $e\in E$ we have
\begin{equation}
	\rk_{M^*(G)}(A\cup \{e\}) - \rk_{M^*(G)}(A) 
		= 1-c_G(E\setminus (A\cup e)) + c_G(E\setminus A).
	\label{eq:rk1}
\end{equation}
We also have
\begin{equation*}
	\rk_{M(G^*)}(A)=|V(G^*)|-c_{G^*}(A).
\end{equation*}
Consider $\g$ as a 1-complex embedded in $\overline{\Sigma}$. Let $\mathcal{A}$ (resp., $\mathcal{E}$) denote the union of all 1-cells in $A$ (resp., $E$), as a subset of $\overline{\Sigma}$. If $X$ is a subspace of $\overline{\Sigma}$, then write $k(X)$ for the number of components of $\overline{\Sigma}\setminus (X\cup v)$, where $v$ is the singular point. We have $c_{G^*}(A)=\rho(\mathcal{E}\setminus \mathcal{A})$. Thus
\begin{equation}
	\rk_{M(G^*)}(A\cup \{e\}) - \rk_{M(G^*)}(A)
		=\rho(\mathcal{E}\setminus \mathcal{A})
		-\rho(\mathcal{E}\setminus (\mathcal{A}\cup e)).
	\label{eq:rk2}
\end{equation}
We claim that \eqref{eq:rk1} is greater than or equal to \eqref{eq:rk2}. Writing $S=E\setminus A$ and letting $\mathcal{S}$ denote the union of elements of $S$, the claim is equivalent to
\begin{equation}
	1+c_G(S)-c_G(S\setminus \{e\}) 
		\geq \rho(\mathcal{S})-\rho(\mathcal{S}\setminus e).
\label{eq:rk3}
\end{equation}
The left side must equal 0 or 1. If it equals 0, then $e$ meets at least one vertex not met by $S$. Thus removing $e$ from $(\overline{\Sigma}\setminus (\mathcal{S}\cup v)) \cup e$ does not increase the number of components, so the right side of \eqref{eq:rk3} is 0, as desired. If instead the left side is 1, then $e$ meets two vertices already met by $S$. Thus removing $e$ from $(\overline{\Sigma}\setminus (\mathcal{S}\cup v)) \cup e$ increases the number of components by at most 1, so the right side of \eqref{eq:rk3} is at most 1, as desired. The claim follows.

Now let $A=S_1\subseteq \cdots \subseteq S_d =B$ such that each $|S_{i+1}\setminus S_i|=1$ for all $i$. The claim gives
\begin{equation*}
	\rk_{M^*(\g)}(S_{i+1}) - \rk_{M^*(\g)}(S_i) 
		\geq \rk_{M(\g^*)}(S_{i+1})-\rk_{M(\g^*)}(S_i).
\end{equation*}
Taking the sum over all $i$ gives \eqref{eq:rk0}.
\end{proof}
\label{lem:quot}
\end{lem}

We will freely refer to the types of circuits and cocircuits of $M(D)$ described in Propositions \ref{prop:circuits} and \ref{prop:cocircuits}, respectively.

\begin{proof}[Proof of Theorem \ref{thm:dualintro}]
Suppose that $D$ is embedded in $\Sigma$. Let $\og$ be the graph obtained from $\g$ by identifying all vertices in $\B$ as a single vertex $v_0$. The embedding of $D$ in $\Sigma$ gives an embedding of $\og$ in $\overline{\Sigma}$ with $v_0$ at the singular point of $\overline{\Sigma}$. Lemma \ref{lem:quot} gives a matroid quotient $M^*(\og)\to M(\og^*)$. It follows that any circuit of $M^*(\og)$, i.e. any cocircuit of $M(D)$ of type (i), is a union of circuits of $M(\og^*)$.

Let $C$ be a circuit of $M(\og^*)$. We claim that $C$ is a union of circuits of $M(D^*)$. Let $H$ be the graph underlying $D^*$. There is an obvious isomorphism $H\to \og^*$ taking the boundary nodes of $D^*$ to the vertices of $\og^*$ in faces of $\og$ whose closures contain $v_0$. Let $k$ be the number of such vertices of $\og^*$ met by $C$. If $k\leq 1$, then $C$ is a circuit of $M(D^*)$ of type (A). If $k\geq 2$, then $C$ is a union of $\lceil k/2\rceil$ circuits of $M(D^*)$ of type (C). The claim follows. Hence every cocircuit of $M(D)$ of type (i) is a union of circuits of $M(D^*)$.

Now suppose that $B\cup \eh$ is a cocircuit of $M(D)$ of type (ii). Note that $B$ is a minimal edge set containing paths between every boundary node of $D^*$. Thus $B$ is a union of crossings of $D^*$. Each of these crossings corresponds to a circuit of $M(D^*)$ of type (A), the union of which is $B\cup \eh$.
\end{proof}

\subsection{Bounds for circular networks}

We give tight upper and lower bounds for the number of circuits of $M(D^*)$ needed to form a cocircuit of $M(D^*)$ when $D$ is circular. Asymptotically, this number is $\Theta(m)$.

\begin{thm}
If $D$ is circular, then every cocircuit of $M(D)$ can be written as a union of $k$ circuits, where $\frac{1}{4}m + \frac{1}{2} \leq k <\frac{1}{2}m+1$, and these bounds are tight.
	
\begin{proof}
We prove only the upper bound; similar techniques yield the lower bound. Examples \ref{eg:sun} and \ref{eg:bisun} will prove that the bounds are tight.

Suppose that $D$ is circular, so that $\Sigma$ is a closed disk and $\overline{\Sigma}$ is a 2-sphere. Let $\og$ be the graph defined in the proof of Theorem \ref{thm:dualintro}. Lemma \ref{lem:quot} gives a matroid quotient $M^*(\og)\to M(\og^*)$. Since the ranks of the two matroids are equal, this matroid quotient is an isomorphism. Thus if $C$ is a cocircuit of $M(D)$ of type (i), i.e. a circuit of $M^*(\og)$, then $C$ is a circuit of $M(\og^*)$. In the proof of Theorem \ref{thm:dualintro}, it was shown that $C$ is a union of $\lceil k/2\rceil$ circuits of $M(D^*)$, where $k$ is the cardinality of a subset of $\B$. Hence $C$ is a union of at most $\lceil m/2\rceil$  circuits of $M(D^*)$, as desired.
		
Now suppose that $C=B\cup \eh$ is a cocircuit of $M(D)$ of type (ii). To prove the bound, can choose crossings $X_1,\ldots,X_{m-1}$ contained in $C$ such that for each $i=2,\ldots,m-1$ there is exactly one boundary node met by both $X_i$ and $\bigcup_{j < i} X_j$. Let $X= \bigcup_{i=1}^{m-1} X_i$. Note that $X$ contains a path between each pair of boundary nodes. Hence $X=B$ by minimality. If $i\neq j$, then $X_i\cup X_j$ is a circuit of $M(D^*)$ of type (C). If $m$ is odd, then
\begin{equation*}
	B=(X_1\cup X_2) \cup \cdots \cup (X_{m-2}\cup X_{m-1})
\end{equation*}
is a union of $(m-1)/2$ circuits of $M(D^*)$. Since $X_1\cup \eh$ is a circuit of $M(D^*)$ of type (A), we can write $C=B\cup (X_1\cup \eh)$ as a union of $(m+1)/2$ circuits of $M(D^*)$. If $m$ is even, then $B\setminus X_1$ is a union of $(m-2)/2$ circuits of $M(D^*)$. Thus $C=(B\setminus X_1) \cup (X_1\cup \eh)$ is a union of $m/2$ circuits of $M(D^*)$. In either case, the number of circuits needed is less than $(m+2)/2$, as desired.
\end{proof}
\label{thm:duallater}
\end{thm}

\begin{eg}
The networks in Figure \ref{fig:sun} are the \emph{sun networks} on 4, 5 and 6 boundary nodes. We obtain such a network on any number of boundary nodes.
	
\begin{figure}[ht]
	\centering
	\begin{tikzpicture}[scale=1.5]
\def\a{2.65}
\draw[opacity=0.4] (\a,0) circle (1);
\foreach \s in {1,2,3,4,5,6}
{
	\coordinate (a\s) at ($({90+60*\s}:0.6)+(\a,0)$);
	\draw[fill=black] (a\s) circle (1.2pt);
	\coordinate (b\s) at ($({60*\s}:1)+(\a,0)$);
	\draw[fill=black] (b\s) circle (1.2pt);
};
\draw (a1)--(a2)--(a3)--(a4)--(a5)--(a6)--(a1);
\draw (a5) -- (b1)--(a6)--(b2)--(a1)--(b3)--(a2)-- (b4)--(a3) -- (b5)--(a4)--(b6)--(a5);

\draw[opacity=0.4] (-\a,0) circle (1);
\foreach \s in {1,2,3,4}
{
	\coordinate (c\s) at ($({90+360*\s/4}:0.6)-(\a,0)$);
	\draw[fill=black] (c\s) circle (1.2pt);
	\coordinate (d\s) at ($({360*(\s+0.5)/4}:1)-(\a,0)$);
	\draw[fill=black] (d\s) circle (1.2pt);
};
\draw (c1) -- (c2) -- (c3) -- (c4) -- (c1);
\draw (c3) -- (d4) -- (c4) -- (d1) -- (c1) -- (d2) -- (c2) -- (d3) -- (c3);

\draw[opacity=0.4] (0,0) circle (1);
\foreach \s in {1,2,3,4,5}
{
	\coordinate (e\s) at ($({54+360*\s/5}:0.6)$);
	\draw[fill=black] (e\s) circle (1.2pt);
	\coordinate (f\s) at ($({-18+72*(\s+0.5)}:1)$);
	\draw[fill=black] (f\s) circle (1.2pt);
};
\draw (e1) -- (e2) -- (e3) -- (e4) -- (e5) -- (e1);
\draw (e3) -- (f4) -- (e4) -- (f5) -- (e5) -- (f1) -- (e1) -- (f2) -- (e2) -- (f3) -- (e3);
\end{tikzpicture}
	\caption{Three sun networks.}
	\label{fig:sun}
\end{figure}
	
The left side of Figure \ref{fig:sundual} illustrates the sun network $D$ with $m=5$ and its dual $D^*$. On the right side, a bond $B$ of $D$ is highlighted in green. The minimum number of circuits of $M(D^*)$ whose union is $B\cup\eh$ is 3. In a similar fashion we can construct a bond of the sun network on any number $m$ of boundary nodes. The corresponding minimum number of circuits is $\lceil m/2\rceil$.
	
\begin{figure}[ht]
	\centering
	\begin{tikzpicture}[scale=1.5]
\draw[opacity=0.4] (0,0) circle (1);
\foreach \s in {1,2,3,4,5}
{
	\coordinate (a\s) at ($({54+360*\s/5}:0.6)$);
	\coordinate (b\s) at ($({-18+72*(\s+0.5)}:1)$);
	\coordinate (c\s) at ({-18+72*\s}:1);
	\coordinate (d\s) at ({18+72*\s}:0.67);
};
\draw (a1) -- (a2) -- (a3) -- (a4) -- (a5) -- (a1);
\draw (a3) -- (b4) -- (a4) -- (b5) -- (a5) -- (b1) -- (a1) -- (b2) -- (a2) -- (b3) -- (a3);

\coordinate (e) at (0,0);
\draw[loosely dashed] (e)--(d1);
\draw[loosely dashed] (e)--(d2);
\draw[loosely dashed] (e)--(d3);
\draw[loosely dashed] (e)--(d4);
\draw[loosely dashed] (e)--(d5);
\draw[loosely dashed] (c1)--(d1)--(c2)--(d2)--(c3) --(d3)--(c4)--(d4)--(c5)--(d5)--(c1);

\foreach \s in {1,2,3,4,5}
{
	\draw[fill=black] (a\s) circle (1.2pt);
	\draw[fill=black] (b\s) circle (1.2pt);
	\draw[fill=white] (c\s) circle (1.2pt);
	\draw[fill=white] (d\s) circle (1.2pt);
};
\draw[fill=white] (e) circle (1.2pt);

\def\a{2.75}
\draw[opacity=0.4] (\a,0) circle (1);
\foreach \s in {1,2,3,4,5}
{
	\coordinate (f\s) at ($(\a,0)+({54+360*\s/5}:0.6)$);
	\coordinate (g\s) at ($(\a,0)+({-18+72*(\s+0.5)}:1)$);
	\draw[fill=black] (g\s) circle (1.2pt);
	\coordinate (h\s) at ($(\a,0)+({-18+72*\s}:1)$);
	\coordinate (i\s) at ($(\a,0)+({18+72*\s}:0.67)$);
};
\draw (f1) -- (f2) -- (f3) -- (f4) -- (f5) -- (f1);
\draw (f3) -- (g4) -- (f4) -- (g5) -- (f5) -- (g1) -- (f1) -- (g2) -- (f2) -- (g3) -- (f3);

\coordinate (j) at (\a,0);
\draw[loosely dashed] (j)--(i1);
\draw[loosely dashed] (j)--(i2);
\draw[loosely dashed] (j)--(i3);
\draw[loosely dashed] (j)--(i4);
\draw[loosely dashed] (j)--(i5);
\draw[loosely dashed] (h1)--(i1)--(h2);
\draw[line width=0.5mm,green] (h2)--(i2)--(h3) --(i3)--(h4)--(i4)--(h5)--(i5)--(h1);

\foreach \s in {1,2,3,4,5}
{
	\draw[fill=black] (f\s) circle (1.2pt);
	\draw[fill=black] (g\s) circle (1.2pt);
	\draw[line width=0.5mm,green,fill=white] (h\s) circle (1.2pt);
};
\foreach \s in {2,3,4,5}
{
	\draw[line width=0.5mm,green,fill=white] (i\s) circle (1.2pt);
};
\draw[fill=white] (i1) circle (1.2pt);
\draw[fill=white] (j) circle (1.2pt);
\end{tikzpicture}
	\caption{A sun network $D$ and its geometric dual, left;
		a bond of $D$ in green, right.}
	\label{fig:sundual}
\end{figure}

\label{eg:sun}
\end{eg}

\begin{eg}
The networks in Figure \ref{fig:bisun} are the \emph{bisected sun networks} on 4, 6 and 10 boundary nodes. We obtain such a network on any even number of boundary nodes.
	
\begin{figure}[ht]
	\centering
	\begin{tikzpicture}[scale=1.5]
\draw[opacity=0.4] (0,0) circle (1);
\foreach \s in {1,2,3,4,5,6}
{
	\coordinate (a\s) at ({90+60*\s}:0.6);
	\draw[fill=black] (a\s) circle (1.2pt);
	\coordinate (b\s) at ({60*\s}:1);
	\draw[fill=black] (b\s) circle (1.2pt);
};
\draw (a1)--(a2)--(a3)--(a4)--(a5)--(a6)--(a1);
\draw (a3)--(a6);
\draw (a5) -- (b1)--(a6)--(b2)--(a1)--(b3)--(a2)-- (b4)--(a3) -- (b5)--(a4)--(b6)--(a5);

\def\a{2.65}
\draw[opacity=0.4] (\a,0) circle (1);
\foreach \s in {1,2,3,4,5,6,7,8}
{
	\coordinate (c\s) at ($({90+360*\s/8}:0.6)+(\a,0)$);
	\draw[fill=black] (c\s) circle (1.2pt);
	\coordinate (d\s) at ($({360*(\s+0.5)/8}:1)+(\a,0)$);
	\draw[fill=black] (d\s) circle (1.2pt);
};
\draw (c1)--(c2)--(c3)--(c4)--(c5) -- (c6)--(c7)--(c8)--(c1);
\draw (c4)--(c8);
\draw (c7)--(d1)--(c8)--(d2)--(c1)--(d3)--(c2)-- (d4)--(c3)--(d5)--(c4)--(d6)--(c5)--(d7)--(c6) -- (d8) -- (c7);

\draw[opacity=0.4] (-\a,0) circle (1);
\foreach \s in {1,2,3,4}
{
	\coordinate (e\s) at ($({90+360*\s/4}:0.6)-(\a,0)$);
	\draw[fill=black] (e\s) circle (1.2pt);
	\coordinate (f\s) at ($({360*(\s+0.5)/4}:1)-(\a,0)$);
	\draw[fill=black] (f\s) circle (1.2pt);
};
\draw (e1) -- (e2) -- (e3) -- (e4) -- (e1);
\draw (e2) -- (e4);
\draw (e3) -- (f4) -- (e4) -- (f1) -- (e1) -- (f2) -- (e2) -- (f3) -- (e3);
\end{tikzpicture}
	\caption{Three bisected sun networks.}
	\label{fig:bisun}
\end{figure}

The left side of Figure \ref{fig:bisundual} illustrates the bisected sun network $D$ with $m=6$ and its dual $D^*$. On the right side, a bond $B$ of $D$ is highlighted in green. The minimum number of circuits of $M(D^*)$ whose union is $B\cup\eh$ is 2. Similarly we can construct a bond of the bisected sun network on any number $m\equiv 2\pmod{4}$ of boundary nodes. The corresponding minimum number of circuits is $\frac{1}{4}m+\frac{1}{2}$, achieving the lower bound in Theorem \ref{thm:duallater}.
	
\begin{figure}[ht]
	\centering
	\begin{tikzpicture}[scale=1.5]
\draw[opacity=0.4] (0,0) circle (1);
\foreach \s in {1,2,3,4,5,6}
{
	\coordinate (a\s) at ({90+60*\s}:0.6);
	\draw[fill=black] (a\s) circle (1.2pt);
	\coordinate (b\s) at ({60*\s}:1);
	\draw[fill=black] (b\s) circle (1.2pt);
	\coordinate (c\s) at ({30+60*\s}:1);
	\coordinate (d\s) at ({60*\s}:0.67);
	\draw[fill=white] (c\s) circle (1.2pt);
	\draw[fill=white] (d\s) circle (1.2pt);
};
\coordinate (e1) at (-0.2,0);
\coordinate (e2) at (0.2,0);
\draw (a1)--(a2)--(a3)--(a4)--(a5)--(a6)--(a1);
\draw (a3)--(a6);
\draw (a5) -- (b1)--(a6)--(b2)--(a1)--(b3)--(a2)-- (b4)--(a3) -- (b5)--(a4)--(b6)--(a5);
\draw[loosely dashed] (e1)--(e2);
\draw[loosely dashed] (e1)--(d2);
\draw[loosely dashed] (e1)--(d3);
\draw[loosely dashed] (e1)--(d4);
\draw[loosely dashed] (e2)--(d5);
\draw[loosely dashed] (e2)--(d6);
\draw[loosely dashed] (e2)--(d1);
\draw[loosely dashed] (d1)--(c1)--(d2)--(c2)--(d3) -- (c3) -- (d4) -- (c4)--(d5)--(c5)--(d6)--(c6) --(d1);

\foreach \s in {1,2,3,4,5,6}
{
\draw[fill=white] (c\s) circle (1.2pt);
\draw[fill=white] (d\s) circle (1.2pt);
};
\draw[fill=white] (e1) circle (1.2pt);
\draw[fill=white] (e2) circle (1.2pt);

\def\a{2.75}
\draw[opacity=0.4] (\a,0) circle (1);
\foreach \s in {1,2,3,4,5,6}
{
	\coordinate (f\s) at ($(\a,0)+({90+60*\s}:0.6)$);
	\draw[fill=black] (f\s) circle (1.2pt);
	\coordinate (g\s) at ($(\a,0)+({60*\s}:1)$);
	\draw[fill=black] (g\s) circle (1.2pt);
	\coordinate (h\s) at ($(\a,0)+({30+60*\s}:1)$);
	\coordinate (i\s) at ($(\a,0)+({60*\s}:0.67)$);
};
\coordinate (j1) at (\a-0.2,0);
\coordinate (j2) at (\a+0.2,0);
\draw (f1)--(f2)--(f3)--(f4)--(f5)--(f6)--(f1);
\draw (f3)--(f6);
\draw (f5) -- (g1)--(f6)--(g2)--(f1)--(g3)--(f2)-- (g4)--(f3) -- (g5)--(f4)--(g6)--(f5);
\draw[loosely dashed] (i1)--(h1)--(i2)--(h2)--(i3) -- (h3) -- (i4) -- (h4)--(i5)--(h5)--(i6)--(h6) --(i1);

\draw[line width=0.5mm,green] (j1)--(j2);
\draw[line width=0.5mm,green] (j1)--(i2);
\draw[line width=0.5mm,green] (j1)--(i3);
\draw[line width=0.5mm,green] (j1)--(i4);
\draw[line width=0.5mm,green] (j2)--(i5);
\draw[line width=0.5mm,green] (j2)--(i6);
\draw[line width=0.5mm,green] (j2)--(i1);
\draw[line width=0.5mm,green] (i1) --(h1);
\draw[line width=0.5mm,green] (i2) --(h2);
\draw[line width=0.5mm,green] (i3) --(h3);
\draw[line width=0.5mm,green] (i4) --(h4);
\draw[line width=0.5mm,green] (i5) --(h5);
\draw[line width=0.5mm,green] (i6) --(h6);

\foreach \s in {1,2,3,4,5,6}
{
	\draw[line width=0.5mm,green,fill=white] (h\s) circle (1.2pt);
	\draw[line width=0.5mm,green,fill=white] (i\s) circle (1.2pt);
};
\draw[line width=0.5mm,green,fill=white] (j1) circle (1.2pt);
\draw[line width=0.5mm,green,fill=white] (j2) circle (1.2pt);
\end{tikzpicture}
	\caption{A bisected sun network $D$ and its geometric dual, left;
		a bond of $D$ in green, right.}
	\label{fig:bisundual}
\end{figure}

\label{eg:bisun}
\end{eg}
\section{Bergman fans}
\label{sec:berg}

Let $M$ be a matroid on $E$. The \emph{tropical linear space} of $M$ is a geometric object that can be used to prove purely combinatorial results about $M$. For references, see \cite{maclagan2015}.

\begin{mydef}
The \emph{tropical linear space} of $M$ is the set $\trop{M}$ consisting of all $x\in \R^E$ such that the minimum of $\{x_e : e\in C\}$ is achieved at least twice for each circuit $C$ of $M$.
\label{def:trop}
\end{mydef}

A \emph{polyhedral fan} is a polyhedral complex whose elements are cones. The set $\trop{M}$ is the support of a polyhedral fan called the \emph{Bergman fan} of $M$.

\begin{mydef}
The \emph{Bergman fan} $\berg{M}$ of $M$ is the coarsest polyhedral fan supported on $\trop{M}$.
\label{def:berg}
\end{mydef}

For each $x\in \R^E$, let $M_x$ be the set of \emph{$x$-maximal bases} of $M$, i.e. the bases that maximize the linear form $\sum_{e\in E} x_e$.

\begin{prop}
Two points $x,y\in \trop{M}$ belong to the same cone of $\berg{M}$ if and only if $M_x=M_y$.
\label{prop:xmax}
\end{prop}

\subsection{Proof of Theorem \ref{thm:bergintro}}

Assume that $G$ is loopless. We write $\trop{D} = \trop{M(D)}$, $\berg{D}=\berg{M(D)}$ and similarly for graphs. Let $C\subseteq V$ be a clique of $G$. We write $K_C$ for the subgraph induced by $C$, and $E(C)$ for its edge set. Given $x\in \R^E$ and $S\subseteq E$, let $x_S$ denote the restriction of $x$ to $S$.

\begin{lem}
Let $C$ be a clique of $G$. For any $x\in \trop{\g}$ and any $x_{E(C)}$-maximal spanning tree $T$ of $K_C$, there is an $x$-maximal spanning tree of $\g$ containing $T$.
	
\begin{proof}
Suppose without loss of generality that $\g$ is simple. Let $w\in \trop{\g}$. We construct the desired tree with a greedy algorithm. Suppose that some number, possibly zero, of the edges of $G$ are colored red. Let $Z$ be a cycle of $\g$. If $Z$ contains a red edge, then do nothing. If $Z$ contains no red edges, then color an $x$-minimal edge of $Z$ red. This procedure is called the \emph{red rule}. Starting with all edges of $\g$ uncolored and applying the red rule to all cycles of $\g$ in any order yields an $x$-maximal spanning tree of $\g$ consisting of the uncolored edges \cite[Theorem 6.1]{tarjan1983}.
		
Suppose that no edge in $Z$ is red, and that $Z$ contains exactly one edge in $E(C)$. Since $x\in \trop{\g}$, there must be an $x$-minimal edge of $Z$ in $E\setminus E(C)$. When applying the red rule to such a cycle, we require that an edge in $E\setminus E(C)$ must be colored red. We call this the \emph{modified red rule}.
		
Start with $\g$ uncolored. First apply the red rule to all cycles of $\g$ contained in $E\setminus E(C)$. Next, apply the modified red rule to all cycles of $\g$ containing exactly one edge in $E(C)$. Every cycle contained in $E\setminus E(C)$ now contains a red edge. Moreover there are no red edges in $E(C)$. Let $S$ be the set of uncolored edges in $E\setminus E(C)$. If $T$ is any $x$-maximal spanning tree of $K_C$, then it follows that $S\cup T$ is an $x$-maximal spanning tree of $\g$.
\end{proof}
\label{new:lem:treecomp}
\end{lem}

Let $U_C\subseteq \R^E$ denote the set of points that are constant on $E(C)$. Let
\begin{equation*}
	\berg{G}_C = \{P\cap U_C : P\in \berg{G} \}.
\end{equation*}
In other words, $\berg{G}_C$ is the restriction of $\berg{G}$ to $U_C$.

\begin{lem}
Let $C$ be a clique of $G$. Every cone of $\berg{G}$ is either contained in $U_C$ or disjoint from $U_C$. In other words, $\berg{G}_C$ is a subfan of $\berg{G}$.

\begin{proof}
Let $x\in \trop{G}$. We have $x_{E(C)}\in\trop{K_C}$, since every cycle of $K_C$ is a cycle of $G$. Propositions \ref{prop:xmax} and \ref{prop:ardila} give correspondences between the cones of $\berg{K_C}$, combinatorial types of phylogenetic trees, and sets of spanning trees of $K_C$. Combining these with \eqref{eq:fform} yields the following equivalent statements, given a point $x\in \trop{G}$:
\begin{enumerate}[(i)]
	\item $x\in U_C$
	\item Every spanning tree of $K_C$ is $x_{E(C)}$-maximal
	\item The phylogenetic tree determined by $x_{E(C)}$ has no internal edges.	\end{enumerate}
	
Suppose that $x\in \trop{G}\cap U_C$ and $y\in \trop{G}\setminus U_C$. Statements (i) and (ii) above give a spanning tree of $K_C$ that is not $y_{E(C)}$-maximal. There is no $y$-maximal spanning tree of $G$ containing $T$. However, $T$ is $x_{E(C)}$-maximal, so Lemma \ref{new:lem:treecomp} gives an $x$-maximal spanning tree of $G$ containing $T$. It follows that $M(G)_x\neq M(G)_y$ in the notation of Proposition \ref{prop:xmax}, so $x$ and $y$ belong to different cones of $\berg{G}$. The result follows.
\end{proof}
\label{lem:subfan}
\end{lem}

Recall that $M(D)$ is simple if and only if $D$ is simple, i.e. if $G$ is simple and $\B$ induces an edgeless subgraph (Proposition \ref{prop:simp}). We reduce the proof of Theorem \ref{thm:bergintro} to the case where $D$ is simple. 

\begin{lem}
It suffices to to prove Theorem \ref{thm:bergintro} when $D$ is simple.
\begin{proof}
Let $P$ be a parallel class of a loopless matroid $M$. Each pair of distinct elements of $P$ forms a circuit. Any element of $\trop{M}$ is constant on such a pair, hence constant on $P$. Thus if $M'$ is the simplification of $M$, then omitting all but one coordinate from each parallel class gives an isomorphism
\begin{equation}
	\berg{M}\to\berg{M'}.
	\label{eq:bergiso}
\end{equation}

Let $D$ be a network and $D'$ its simplification, obtained by deleting all duplicate edges and edges between boundary nodes. Let $H$ be the graph obtained from $D'$ by adding an edge between each pair of boundary nodes. Note that $H$ is also the simplification of $\mg$. From \eqref{eq:bergiso} we obtain isomorphisms
\begin{equation}
	\berg{D}\to \berg{D'}
	\label{eq:iso1}
\end{equation}
and $\berg{\mg}\to \berg{H}$. Applying Lemma \ref{lem:subfan}, the latter restricts to an isomorphism
\begin{equation}
	\berg{\mg}_\B \to \berg{H}_\B.
	\label{eq:iso2}
\end{equation}
Theorem \ref{thm:bergintro} for simple networks gives an isomorphism
\begin{equation}
	\berg{D'}\to\berg{H}_{\B}.
	\label{eq:iso3}
\end{equation}
Combining \eqref{eq:iso1}, \eqref{eq:iso2} and \eqref{eq:iso3} gives the desired isomorphism $\trop{D}\to\trop{\mg}_{\B}$.
\end{proof}
\label{lem:bergsimp}
\end{lem}

\begin{proof}[Proof of Theorem \ref{thm:bergintro}]
Lemma \ref{lem:bergsimp} lets us assume that $D$ is simple. Recall that $\eo = E\cup \eh$ is the ground set of $M(D)$. We write the edge set of $\mg$ as $E(\mg)= E\cup E(\B)$. Let $f:\R^{\eo}\to \R^{E(\mg)}$ be given by
\begin{equation*}
	f(x)_e =
	\begin{cases}
		x_e & \mbox{if } e\in E\\
		x_{\eh} &\mbox{if } e\in E(\B).
	\end{cases}
\end{equation*}

Let $\trop{\mg}_{\B}\subseteq \trop{\mg}$ denote the set of points that are constant on $E(\B)$. Let $x\in \trop{D}$, and let $Z$ be a cycle of $\mg$. We claim that the minimum coordinate of $f(x)_Z$ is achieved at least twice, i.e. that $f(x)\in \trop{\mg}_{\B}$. We argue three cases:
\begin{enumerate}[(i)]
	\item $Z\subseteq E$
	\item $Z\subseteq E(\B)$
	\item $Z$ is not contained in $E$ or $E(\B)$.
\end{enumerate}
In case (i), $Z$ is a circuit of $M(D)$, so the claim holds. In case (ii), $f(x)_Z$ is constant, so the claim holds. In case (iii), $Z\cap E$ is a union of crossings of $D$. In particular, there is a crossing $C\subseteq Z\cap E$ such that the minimum coordinate of $f(x)_Z$ occurs on $C$ or on $E(\B)$. Since $C\cup\eh$ is a circuit of $M(D)$, the minimum coordinate of $x_{C\cup \eh}$ occurs at least twice. This equals the minimum coordinate of $f(X)_Z$, so the claim follows in the final case.

Clearly $f$ is linear and injective; thus with the claim proven, we have
\begin{equation*}
	f(\trop{D}) = \trop{\mg}_{\B}.
\end{equation*}
Since $\berg{D}$ is the coarsest polyhedral fan supported on $\trop{D}$, the set $\{f(P) : P\in \berg{D}\}$ is the coarsest polyhedral fan supported on $\trop{\mg}_{\B}$. Lemma \ref{lem:subfan} says that $\berg{\mg}_\B$ is a subfan of $\berg{\mg}$, which must be the coarsest polyhedral fan supported on $\trop{\mg}_{\B}$. It follows that $\berg{\mg}_{\B} = \{f(P) : P\in \berg{D}\}$, so $f$ is an isomorphism $\berg{D}\to \berg{\mg}_{\B}$ as desired.
\end{proof}

\subsection{Phylogenetic trees and discriminantal arrangements}
\label{new:sec:phylo}

Let $T$ be a rooted tree with labeled leaves and a real-valued function $\omega$ on its edges. Suppose that the root is not a leaf, and that no non-root vertex has degree 2. The \emph{distance} between distinct vertices $i$ and $j$ of $T$, denoted by $d_T(i,j)$, is the (possibly negative) sum of $\omega(e)$ over the edges $e$ in the unique path between $i$ and $j$. An edge of $T$ is \emph{internal} if it is not incident to a leaf.

\begin{mydef}
The pair $(T,\omega)$ is called a \emph{phylogenetic tree} if the following hold:
\begin{enumerate}[(i)]
	\item The distance between the root and any leaf is the same
	\item $\omega(e)>0$ for every internal edge $e$.
\end{enumerate}
A phylogenetic tree with $n$ leaves is called a \emph{phylogenetic $n$-tree}.
\end{mydef}

The vertices of a phylogenetic tree form a poset in which the root is the unique minimal element and the leaves are the maximal elements. If two vertices $i$ and $j$ are adjacent with $i\leq j$, then $j$ is the \emph{child} of $i$. A phylogenetic tree is \emph{binary} if every non-leaf vertex has exactly two children. The \emph{most recent common ancestor} of two vertices $i$ and $j$ is their infimum. The \emph{combinatorial type} of a phylogenetic tree $(T,\omega)$ is simply the tree $T$ along with its root and leaf labeling.

There is a topological space $\mathcal{T}_n$, introduced in \cite{billera2001}, that realizes the space of phylogenetic $n$-trees and supports a polyhedral fan. We let $\mathcal{T}_n$ denote both the polyhedral fan and the underlying topological space.

To construct $\mathcal{T}_n$, first take one $(n-2)$-dimensional orthant for each combinatorial type of binary phylogenetic $n$-trees. These orthants are the maximal cones of $\mathcal{T}_n$. The facets of each orthant correspond to the internal edges of the binary tree; a point in a facet represents a tree in which the corresponding edge has been contracted. Glue the facets of any 2 orthants together when they represent the same combinatorial type of tree. The lower-dimensional faces represent further contractions; glue them together whenever they represent the same combinatorial type of tree. The common vertex of the orthants represents the tree with no internal edges. See Figure \ref{fig:t3space} for an illustration of $\mathcal{T}_3$.

\begin{figure}[ht]
	\centering
	\begin{tikzpicture}[scale=0.9]
\coordinate (p0) at (0,0);
\coordinate (p1) at (6.5,0);
\coordinate (p2) at (-5.5,3.5);
\coordinate (p3) at (-5.5,-3.5);
\draw[thick] (p1) -- (p0) -- (p2);
\draw[thick] (p0) -- (p3);

\def\xl{-5}
\def\yl{-0.8}
\coordinate (tl1) at ($(0,0)+(\xl,\yl)$);
\coordinate (tl2) at ($(1,0)+(\xl,\yl)$);
\coordinate (tl3) at ($(2,0)+(\xl,\yl)$);
\coordinate (tli) at ($(1.5,0.8)+(\xl,\yl)$);
\coordinate (tlr) at ($(1,1.6)+(\xl,\yl)$);
\draw (tl1) node[below=1mm] {$1$};
\draw (tl2) node[below=1mm] {$2$};
\draw (tl3) node[below=1mm] {$3$};
\draw (tl1) -- (tlr) -- (tl3);
\draw (tl2) -- (tli);
\draw[fill=black] (tl1) circle (1.5pt);
\draw[fill=black] (tl2) circle (1.5pt);
\draw[fill=black] (tl3) circle (1.5pt);
\draw[fill=black] (tli) circle (1.5pt);
\draw[fill=black] (tlr) circle (1.5pt);

\def\xb{-0.6}
\def\yb{-3.1}
\coordinate (tb1) at ($(0,0)+(\xb,\yb)$);
\coordinate (tb2) at ($(1,0)+(\xb,\yb)$);
\coordinate (tb3) at ($(2,0)+(\xb,\yb)$);
\coordinate (tbi) at ($(1.5,0.8)+(\xb,\yb)$);
\coordinate (tbr) at ($(1,1.6)+(\xb,\yb)$);
\draw (tb1) node[below=1mm] {$3$};
\draw (tb2) node[below=1mm] {$1$};
\draw (tb3) node[below=1mm] {$2$};
\draw (tb1) -- (tbr) -- (tb3);
\draw (tb2) -- (tbi);
\draw[fill=black] (tb1) circle (1.5pt);
\draw[fill=black] (tb2) circle (1.5pt);
\draw[fill=black] (tb3) circle (1.5pt);
\draw[fill=black] (tbi) circle (1.5pt);
\draw[fill=black] (tbr) circle (1.5pt);

\def\xt{2.8}
\def\yt{-2.5}
\coordinate (tt1) at ($(0,0)+(\xt,\yt)$);
\coordinate (tt2) at ($(1,0)+(\xt,\yt)$);
\coordinate (tt3) at ($(2,0)+(\xt,\yt)$);
\coordinate (tti) at ($(1.5,0.8)+(\xt,\yt)$);
\coordinate (ttr) at ($(1,1.6)+(\xt,\yt)$);
\draw (tt1) node[below=1mm] {$2$};
\draw (tt2) node[below=1mm] {$1$};
\draw (tt3) node[below=1mm] {$3$};
\draw (tt1) -- (ttr) -- (tt3);
\draw (tt2) -- (tti);
\draw[fill=black] (tt1) circle (1.5pt);
\draw[fill=black] (tt2) circle (1.5pt);
\draw[fill=black] (tt3) circle (1.5pt);
\draw[fill=black] (tti) circle (1.5pt);
\draw[fill=black] (ttr) circle (1.5pt);

\def\xd{1.7}
\def\yd{1.3}
\coordinate (td1) at ($(0,0)+(\xd,\yd)$);
\coordinate (td2) at ($(1,0)+(\xd,\yd)$);
\coordinate (td3) at ($(2,0)+(\xd,\yd)$);
\coordinate (tdr) at ($(1,1.2)+(\xd,\yd)$);
\draw (td1) node[below=1mm] {$1$};
\draw (td2) node[below=1mm] {$2$};
\draw (td3) node[below=1mm] {$3$};
\draw (td1) -- (tdr) -- (td3);
\draw (td2) -- (tdr);
\draw[fill=black] (td1) circle (1.5pt);
\draw[fill=black] (td2) circle (1.5pt);
\draw[fill=black] (td3) circle (1.5pt);
\draw[fill=black] (tdr) circle (1.5pt);

\draw[->] ($(tli)+(0.5,0)$) to [bend right]  ($0.7*(p0)+0.3*(p2) - (0.2,0.2)$);
\draw[->] ($0.5*(tbr)+0.5*(tb1)+(-0.5,0)$) to [bend left]  ($0.65*(p0)+0.35*(p3) + (0.2,-0.2)$);
\draw[->] ($(tti)+(0.5,0)$) to [bend right]  ($0.15*(p0)+0.85*(p1) + (0,-0.3)$);
\draw[->] ($0.5*(td1)+0.5*(tdr)+(-0.5,0)$) to [bend right]  ($(p0) + (0.1,0.3)$);
\end{tikzpicture}
	\caption{The polyhedral fan $\mathcal{T}_3$.}
	\label{fig:t3space}
\end{figure}

The cones of $\mathcal{T}_n$ correspond to the combinatorial types of phylogenetic $n$-trees. The product $\mathcal{T}_n\times \R$ supports a polyhedral fan inherited from $\mathcal{T}_n$; the factor $\R$ keeps track of the distance between the root and leaves, which is not fixed in the construction of $\mathcal{T}_n$.

\begin{prop}[{\cite[Proposition 3]{ardila2006}}]
There is an isomorphism of polyhedral fans
\begin{equation}
	f:\mathcal{T}_n\times \R\to \berg{K_n}.
	\label{eq:f}
\end{equation}
\label{prop:ardila}
\end{prop}

If we identify the leaves of a phylogenetic $n$-tree $(T,\omega)$ with the vertices of $K_n$, then the coordinates of $\trop{K_n}$ are indexed by pairs of leaves, and the piecewise-linear homeomorphism in \eqref{eq:f} is easy to describe:
\begin{equation}
	f(T,\omega)_{ij}=d_T(i,j)
	\label{eq:fform}
\end{equation}
for all leaves $i$ and $j$ of $T$. We will use this to obtain a version of Proposition \ref{prop:ardila} for ``complete'' networks.

\begin{mydef}
A set $S$ of leaves of $T$ is \emph{equidistant} if every pair of leaves in $S$ has the same most recent common ancestor. In other words, $S$ is equidistant if $d_T(i,j)=d_T(i,k)$ for all $i,j,k\in S$.
\end{mydef}

\begin{eg}
Consider the phylogenetic trees illustrated in Figure \ref{fig:equiphy}. Both trees have equidistant 3-sets (i.e., equidistant sets $S$ with $|S|=3$) marked in white.
\end{eg}

\begin{figure}[ht]
\centering
\begin{tikzpicture}[scale=2.5]
\coordinate (1l1) at (0,0);
\coordinate (1l2) at (0.4,0);
\coordinate (1l3) at (0.6,0);
\coordinate (1l4) at (1,0);
\coordinate (1l5) at (1.2,0);
\coordinate (1l6) at (1.6,0);
\coordinate (1l7) at (2,0);
\coordinate (2l1) at (0.2,0.2);
\coordinate (2l2) at (0.8,0.2);
\coordinate (2l3) at (1.4,0.2);
\coordinate (3l1) at (0.8,0.8);
\coordinate (4l1) at (1,1);

\draw (1l1) -- (4l1) -- (1l7);
\draw (1l2)--(2l1);
\draw (1l3)--(2l2)--(1l4);
\draw (2l2)--(3l1);
\draw (1l6)--(3l1);
\draw (1l5)--(2l3);

\draw[fill=white] (1l1) circle (1pt);
\draw[fill=white] (1l3) circle (1pt);
\draw[fill=white] (1l5) circle (1pt);
\draw[fill=black] (1l2) circle (1pt);
\draw[fill=black] (1l4) circle (1pt);
\draw[fill=black] (1l6) circle (1pt);
\draw[fill=black] (1l7) circle (1pt);
\draw[fill=black] (2l1) circle (0.8pt);
\draw[fill=black] (2l2) circle (0.8pt);
\draw[fill=black] (2l3) circle (0.8pt);
\draw[fill=black] (3l1) circle (0.8pt);
\draw[fill=black] (4l1) circle (0.8pt);

\def\a{2.5}
\coordinate (2-1l1) at ($(\a,0)+(0,0)$);
\coordinate (2-1l2) at ($(\a,0)+(0.66,0)$);
\coordinate (2-1l3) at ($(\a,0)+(1,0)$);
\coordinate (2-1l4) at ($(\a,0)+(2,0)$);
\coordinate (2-2l1) at ($(\a,0)+(0.33,0.33)$);
\coordinate (2-3l1) at ($(\a,0)+(1,1)$);

\draw (2-1l1) -- (2-3l1) -- (2-1l4);
\draw (2-2l1) -- (2-1l2);
\draw (2-1l3) -- (2-3l1);

\draw[fill=white] (2-1l1) circle (1pt);
\draw[fill=white] (2-1l3) circle (1pt);
\draw[fill=white] (2-1l4) circle (1pt);
\draw[fill=black] (2-1l2) circle (1pt);
\draw[fill=black] (2-2l1) circle (0.8pt);
\draw[fill=black] (2-3l1) circle (0.8pt);
\end{tikzpicture}
\caption{Two phylogenetic trees with equidistant 3-sets marked in white.}
\label{fig:equiphy}
\end{figure}

We consider phylogenetic trees with a prescribed equidistant $m$-set. Up to permutation of coordinates, this space depends only on the size of the equidistant set.

\begin{mydef}
Let $\mathcal{T}_{m,n}\times\R$ denote the subfan of $\mathcal{T}_{m+n}\times \R$ of phylogenetic $(m+n)$-trees with a prescribed equidistant $m$-set.
\end{mydef}

Let $D=D_{m,n}$ as in Example \ref{eg:dmn}, so that $\mg=K_{m+n}$. Restricting the map
\begin{equation*}
	f:\mathcal{T}_{m+n}\times \R\to \trop{\mg}
\end{equation*}
from \eqref{eq:f} to the set of phylogenetic $(m+n)$-trees with $\B$ equidistant, we obtain:

\begin{prop}
There is an isomorphism of polyhedral fans
\begin{equation}
	\mathcal{T}_{m,n}\times\R\to \berg{D_{m,n}}.
\end{equation}
\label{prop:newphylo}
\end{prop}

\subsection{A minimal tropical basis}

For any set $S\subseteq E$, let $V(S)$ be the set of points $x\in \R^E$ such that the minimum coordinate of $x_S$ is achieved at least twice. This set is the \emph{tropical hyperplane} defined by $S$. For any set $\mathcal{S}$ of subsets of $T$, let $V(\mathcal{S})=\bigcap_{S\in \mathcal{S}} V(S)$. For example, we have $\trop{M} = V(\mathcal{C})$, where $\mathcal{C}$ is the set of circuits of $M$.

\begin{mydef}
A set $\mathcal{S}$ of subsets of $E$ is a \emph{tropical basis} of $M$ if $V(\mathcal{S})=\trop{M}$.
\end{mydef}

In \cite{yu2007} the problem of computing a minimal tropical basis of a given matroid was posed, and it was shown that graphic matroids admit a unique minimal tropical basis. We compute a minimal tropical basis of $M(D)$. To \emph{paste} two sets is to take their symmetric difference. Since the symmetric difference operation is commutative and associative, we can paste any finite number of sets.

\begin{lem}
Let $T\subset E$, and let $\mathcal{S}$ be a set of subsets of $E$. If $T$ is obtained by pasting elements of $\mathcal{S}$, then $V(\mathcal{S})\subseteq V(T)$.

\begin{proof}
Suppose that $T$ is obtained by pasting elements $S_1$ and $S_2$ of $\mathcal{S}$. Let $x\in V(\mathcal{S})$. Suppose without loss of generality that $\min\{x_e :e\in S_1\}\leq \min \{x_e : e\in S_2\}$. If $\min\{x_e :e\in S_1\}$ is achieved on $S_1\cap S_2$, then it equals $\min\{x_e :e\in S_2\}$, so $\min\{x_e : e\in T\}$ is achieved at least once on each of $S_1$ and $S_2$. If not, then $\min\{x_e : e\in T\}$ is achieved at least twice on $S_1$. In either case we have $x\in V(T)$.
\end{proof}
\label{lem:pasting}
\end{lem}

A \emph{chord} of a circuit $C$ is any element $i$ such that there exist circuits $C_1$ and $C_2$ with $C_1\cap C_2 = i$ and $C_1\triangle C_2 = C$. Recall the types of circuits of $M(D)$ from Proposition \ref{prop:circuits}. A chord of a circuit $X\cup \eh$ of type (A) is any edge in $E\setminus X$ joining two vertices met by $X$. A chord of a circuit $C$ of type (B) is any edge in $E\setminus C$ joining two vertices met by $C$.

\begin{prop}
If $M(D)$ is simple, then there is a minimal tropical basis of $M(D)$ consisting of all chordless circuits of types (A) and (B) in Proposition \ref{prop:circuits}.
	
\begin{proof}
Suppose that a circuit $C\subseteq E$ of $M(D)$ of type (B) admits a chord $j$. Then there is a set $F\subseteq C$ such that $F\cup j$ and $(C\setminus F)\cup j$ are circuits of type (B). Pasting these two circuits yields $C$. Iterating this argument gives $C$ as a pasting of chordless cycles of type (B).
		
Let $\overline{X}=X\cup \eh$ be a circuit of type (A) for some $X\subseteq E$. If $i$ is a chord of $\overline{X}$, then $X\cup i$ contains a single circuit $Z$ of type (B), and the set $(X\setminus Z)\cup i$ is a crossing. Pasting $Z$ and $(X\setminus Z)\cup i$ yields $\overline{X}$. Iterating this argument and the argument from the first paragraph gives $\overline{X}$ as a pasting of chordless cycles of types (A) and (B).
		
Let $\mathcal{B}$ be the set of all chordless circuits of types (A) and (B). Let $Y$ be a circuit of type (C).  We claim that $V(\mathcal{B}) \subseteq V(Y)$. Let $C_1$ and $C_2$ be distinct crossings contained in $Y$, so that $Y=C_1\cup C_2$. Let $x \in V(\mathcal{B})$. Suppose without loss of generality that $\min\{x_e : e \in Y \cup \eh\}$ occurs on $C_1\cup\eh$. If $C_1$ and $C_2$ are disjoint and $x_{\eh}=\min\{x_e : e \in Y\}$, then this minimum is achieved at least once on each of $C_1$ and $C_2$. If $C_1$ and $C_2$ are disjoint and $x_{\eh} \neq \min\{x_e : e \in Y\}$, then this minimum is achieved at least twice on $C_1$. Therefore $x \in V(Y)$ in this case.
		
Suppose now that $C_1$ and $C_2$ are not disjoint, so that $Y$ contains a third crossing $C_3$. If $\min\{x_e : e \in C_2\} > \min\{x_e : e \in C_1 \cup \eh\}$, then the latter must occur twice on $C_1$; otherwise, $\min\{x_e : C_2\cup\eh\}$ occurs only once, on $\eh$, contradicting $x\in V(\mathcal{B})$. Hence $\min\{x_e : e\in Y\}$ occurs twice, on $C_1$. If $\min\{x_e : e \in C_2\}=\min\{x_e : e \in C_1 \cup \eh\}$ and these are both equal to $x_{\eh}$, then $\min\{x_e:e\in C_3\cup\eh\}$ occurs only once, on $\eh$, a contradiction. Hence if $\min\{x_e : e\in C_2\}=\min\{x_e:e\in C_1\cup\eh\}$, then these minima are less than $x_{\eh}$, and the minimum $\min\{x_e : e\in Y\}$ occurs at least 3 times. Therefore $x\in V(Y)$ again, proving the claim. It follows that $\mathcal{B}$ is a tropical basis of $M(D)$.
		
We show that $\mathcal{B}$ is minimal. Suppose that the circuit $C$ from above is chordless. For some $e\in C$, let $y \in \R^{\eo}$ be 1 on $C\setminus e$ and 0 on the rest of $\eo$. Any circuit in $\mathcal{B} \setminus C$ must contain at least two elements of $\eo$ not in $C$. The point $y$ achieves its minimum at least twice on such a circuit. Hence $y \in V(\mathcal{B}\setminus C)\setminus V(\mathcal{B})$, proving that $\mathcal{B}\setminus C$ is not a tropical basis.
		
Suppose now that the circuit $\overline{X}$ from above is chordless. Let $z\in \R^{\eo}$ be 1 on $X$ and 0 on $\eo\setminus X$. Any circuit of type (A) in $\mathcal{B}\setminus \overline{X}$ must contain $\eh$ and at least one edge in $E\setminus X$. Any circuit of type (B) in $\mathcal{B}\setminus \overline{X}$ must contain at least two elements of $E\setminus X$. The point $z$ achieves its minimum at least twice on any such circuit. Hence $\mathcal{B}\setminus \overline{X}$ is not a tropical basis, proving that $\mathcal{B}$ is minimal.
\end{proof}
\label{prop:mintrop}
\end{prop}
\section{Response matrices and the half-plane property}
\label{sec:hpp}

In this section, we show that every Dirichlet matroid has the half-plane property and use this to prove Theorem \ref{thm:realintro}. Let $\mathcal{S}$ be a finite set of finite sets. The \emph{generating polynomial} of $\mathcal{S}$ is the sum of the monomials $\prod_{e\in S} x_e$ over all $S\in \mathcal{S}$. The \emph{basis generating polynomial} of a matroid $M$ is the generating polynomial of the set of bases of $M$.

Let $\Re$ and $\Im$ denote the real and imaginary part operators, respectively. Write
\begin{equation*}
	\R_+^n=(0,\infty)^n.
\end{equation*}
A polynomial $f\in \C[x_1,\ldots,x_n]$ is \emph{stable} if $f$ has no zeros $\xx$ with $\Im(\xx)\in\R^n_+$.

\begin{mydef}
A matroid $M$ is \emph{HPP} (short for \emph{half-plane property}) if the basis generating polynomial of $M$ is stable.
\label{def:hpp}
\end{mydef}

For a list of known HPP and non-HPP matroids, see \cite{dleon2009}. The next fact, part of the ``folklore'' of electrical engineering, describes a fundamental family of examples \cite[p. 4]{wagner2005}.

\begin{prop}
Every graphic arrangement is HPP.
\label{prop:graphstable}
\end{prop}

\subsection{Laplacian and response matrices}

The \emph{admittance} of an edge $e\in E$ is a complex number that measures how easily current may flow through $e$. In this section, we fix admittances $x\in \C^E$. We assume that $D$ is simple, since parallel edges can be treated as a single edge whose admittance is the sum of the individual admittances, and since current does not flow through loops or edges between boundary nodes.

Let $L=L(\xx)$ denote the $V\times V$ \emph{weighted Laplacian matrix}, given by
\begin{equation*}L_{ij}=\begin{cases}\sum_{k\sim i} x_{ik}&\mbox{if }i=j\\ -x_{ij}&\mbox{if }i\sim j\\ 0&\mbox{else}.\end{cases}\end{equation*}
Write $L$ in block form as
\begin{equation}
	L=
	\begin{bmatrix}
		A   & B\\
		B^T & C
	\end{bmatrix},
	\label{eq:lform}
\end{equation}
where the rows and columns of $A$ are indexed by $\B$. If $C$ is invertible, then the \emph{response matrix} of $D$ is the Schur complement $L/C$, given by
\begin{equation}
	\rs=A-BC^{-1}B^T.
\end{equation}
If voltages $u\in \R^\B$ are applied to the boundary of $D$, e.g. by attaching batteries, then $\rs u$ is the vector of resulting currents at the boundary nodes.

\begin{lem}
Suppose that $C$ is invertible. Let $u\in \R^\B$,  $v=-C^{-1}B^Tu$ and $\phi=\rs u$. We have
\begin{equation}
	\begin{bmatrix}
		A   & B\\
		B^T & C
	\end{bmatrix}
	\begin{bmatrix}
		u \\ v
	\end{bmatrix}
	=
	\begin{bmatrix}
		\phi \\ 0
	\end{bmatrix}.
	\label{eq:block}
\end{equation}
\label{lem:harm}
\begin{proof}
This is a direct computation.
\end{proof}
\end{lem}

\begin{lem}
If $\Re(\xx)\in \R^E_+$, then $\Re(\rs)$ is positive semidefinite.

\begin{proof}
Suppose that $\Re(\xx)\in \R^E_+$, and let $u\in \R^\B$. Let $f\in \C^V$ be the column vector on the left side of \eqref{eq:block}. Order the boundary nodes $1,\ldots,m$ and the interior vertices $m+1,\ldots,d$. We have
\begin{align*}
	u^T \Re(\rs)u &= \sum_{i,j=1}^m u_i \Re(\rs)_{ij} u_j \\
		&= \sum_{i,j=1}^m \Re(u_i \overline{\rs_{ij}u_j})\\
		&= \sum_{i=1}^m \Re(u_i \overline{[\rs u]_i}).
	\label{eq:innerprod}
\end{align*}
Lemma \ref{lem:harm} implies that $Lf|_{\B}=\rs u$ and $Lf|_{V\setminus \B}=0$, so
\begin{equation*}
	\sum_{i=1}^m \Re(u_i \overline{[\rs u]_i}) 
		= \sum_{i=1}^d \Re(f_i \overline{[Lf]_i}).
\end{equation*}
Write $x_{ij} = 0$ for all non-adjacent $i,j\in V$. Direct computation gives
\begin{equation*}
	[Lf]_i = \sum_{j=1}^d x_{ij} (f_i-f_j),
\end{equation*}
so
\begin{align*}
	\sum_{i=1}^d \Re(f_i \overline{[Lf]_i})
		&= \sum_{i,j=1}^d \Re(f_i \overline{x_{ij}(f_i-f_j)})\\
	&=\sum_{1\leq i< j\leq d} \Re(f_i \overline{x_{ij}(f_i-f_j)})
		+\sum_{1\leq j<i\leq d} \Re(f_i \overline{x_{ij}(f_i-f_j)})\\
	&=\sum_{1\leq i< j\leq d} \Re(f_i \overline{x_{ij}(f_i-f_j)})
		-\sum_{1\leq i<j\leq d} \Re(f_j \overline{x_{ij}(f_i-f_j)})\\
	&=\sum_{1\leq i< j\leq d} \Re((f_i-f_j) \overline{x_{ij}(f_i-f_j)})\\
	&=\sum_{1\leq i<j \leq d} \Re(x_{ij})|f_i - f_j|^2
\end{align*}
is positive. The result follows.
\end{proof}
\label{lem:psd}
\end{lem}

\subsection{Basis generating polynomials}

We establish formulas for the basis generating polynomial of $M(D)$ and use them to prove that every Dirichlet arrangement has the half-plane property. This can also be proven using the results of Appendix \ref{sec:printrun}, but the connection to the response matrix is lost.

Let $\pb$ denote the basis generating polynomial of $M(D)$. For $i=0,1$, let $P_i$ denote the generating polynomial of the set $\Sigma_i$ from Definition \ref{def:suncscro}. Proposition \ref{prop:bases} implies that
\begin{equation}
	\pb(\xx,\xh )= \pcro(\xx)+\xh \punc(\xx)
	\label{eq:pb1}
\end{equation}
for all $(\xx,\xh)\in \C^E\times \C$, where $\xh$ is the variable corresponding to $\eh$. Let $\tr\rs$ denote the trace of the response matrix $\rs$.

\begin{lem}
For all $(\xx,\xh )\in \C^E\times \C$ with $\Im(\xx)\in \R^E_+$, the basis generating polynomial of $M(D)$ is given by
\begin{equation}
	\pb(\xx,\xh )=\punc(\xx)\left(\xh +\frac{1}{2}\tr\rs\right).
\end{equation}
	
\begin{proof}
For all distinct boundary nodes $i$ and $j$ let
\begin{equation*}
	\sij = \{ F \in \scro : F \mbox{ contains a path from } i \mbox{ to } j \}.
\end{equation*}
The matrix-tree theorem for principal minors (see e.g. \cite{chaiken1982}) gives $\det C = \punc$, where $C$ is defined in \eqref{eq:lform}. Note that $\punc$ is the basis generating polynomial of $M(\og)$, where $\og$ is the graph obtained by identifying all boundary nodes as a single vertex. Proposition \ref{prop:graphstable} implies that $\rs$ is well defined whenever $\Im(\xx)\in \R^E_+$. Thus if $\Im(\xx)\in \R^E_+$, then for all $i\neq j$ we have
\begin{equation}
	-\rs_{ij}=\frac{P_{ij}}{\punc},
	\label{eq:rs1}
\end{equation}
where $P_{ij}$ is the generating polynomial of $\sij$ (see e.g. \cite[Proposition 2.8]{kenyon2011}). It is not hard to see that $\rs$ is symmetric, and that every row sum of $\rs$ is zero \cite[p. 3]{curtis2000}. Since $\pcro = \frac{1}{2}\sum_{i\neq j} P_{ij}$, we deduce that $\sum_{i\neq j}\rs_{ij}=-\tr\rs$. The result now follows from \eqref{eq:pb1} and \eqref{eq:rs1}.
\end{proof}
\label{lem:pbfor}
\end{lem}

\begin{prop}
Every Dirichlet matroid has the half-plane property.
\begin{proof}
Let $(\xx,\xh)\in \C^E\times \C$ with $\Re(\xx)\in \R^E_+$ and $\Re(\xh)>0$. Since $\pb$ is homogeneous, it suffices to show that $\pb(\xx,\xh)\neq 0$. Since $\punc$ is the basis generating polynomial of $M(\og)$, Proposition \ref{prop:graphstable} implies that $\punc(\xx)\neq 0$. Thus by Lemma \ref{lem:pbfor} it suffices to show that $\Re(\tr\rs(\xx))\geq 0$ whenever $\Re(\xx)\in \R_+^n$. This follows from Lemma \ref{lem:psd}.
\end{proof}
\label{prop:hpp}
\end{prop}

\begin{remark}
Given an edge $e\in E$, the quantities $\Re(\xx_e)$ and $\Im(\xx_e)$ are called the \emph{conductance} and \emph{susceptance} of $e$, respectively. Each edge represents an ideal electrical component whose type can determine the signs of its conductance and susceptance. For example, a \emph{resistor} has positive conductance and zero susceptance. A \emph{capacitor} (resp., \emph{inductor}) has zero conductance and positive (resp., negative) susceptance.
\end{remark}

\subsection{Interlacing zeros and proof of Theorem \ref{thm:realintro}}

For $1\leq i\leq n$, the \emph{Wronskian} with respect to $x_i$ is the bilinear map $W_{x_i}$ on $\R[x_1,\ldots,x_n]$ given by
\begin{equation}
	W_{x_i}(f,g) = f \cdot \partial_i g - \partial_i f\cdot g.
\end{equation}
If $\xx,y\in \R^n$, then $W_t(f(\xx+t\yy),g(\xx+t\yy))$ is a univariate polynomial in $t$. Two polynomials $f,g\in \R[x_1,\ldots,x_n]$ are in \emph{proper position}, written $f\ll g$, if for all $(\xx,\yy)\in \R^n\times\R^n_+$ we have
\begin{equation}
	W_t(f(\xx+t\yy),g(\xx+t\yy))(t) \geq 0
\end{equation}
for all $t\in \R$. For technical reasons we also declare that $0\ll f$ and $f\ll 0$ for all $f$. If $f\ll g$, then for any $(\xx,\yy)\in \R^n\times\R^n_+$ the real zeros of $f(\xx+t\yy)$ and $g(\xx+t\yy)$ \emph{interlace} in the following sense (see \cite{branden2007}).
	
\begin{mydef}
Two finite sets $S,T\subseteq \R$ whose cardinalities differ by at most 1 are said to \emph{interlace} if the elements $s_i$ of $S$ and $t_i$ of $T$ can be indexed so that either $s_1\leq t_1\leq s_2\leq t_2\leq \cdots$ or $t_1\leq s_1\leq t_2\leq s_2\leq \cdots$. More generally, two finite subsets of a line $L\subseteq \R^n$ \emph{interlace} if their images interlace under any homeomorphism $L\to \R$.
\end{mydef}
	
\begin{prop}[{\cite[Corollary 5.5]{branden2007}}]
Let $f,g\in \R[x_1,\ldots,x_n]$. We have $g\ll f$ if and only if $f+\xh g\in \R[x_0,\ldots,x_n]$ is stable.
\label{prop:branpp}
\end{prop}

\begin{proof}[Proof of Theorem \ref{thm:realintro}]
This follows from \eqref{eq:pb1}, Proposition \ref{prop:hpp}, and Proposition \ref{prop:branpp}.
\end{proof}

\begin{eg}
Let $D$ be a star network, and write $m=|\B|$. Writing the conductances as $\mathbf{x}=(x_1,\ldots,x_m)$, we have
\begin{equation*}
	\tr\Lambda = \frac{\sum_{i\neq j} x_ix_j}{\sum_i x_i}.
\end{equation*}
Every line in $\R^m$ with positive direction vector intersects exactly 2 zeros of $\tr\Lambda$ (counted with multiplicity) and one pole in between. For example, if $\mathbf{x}=(-1,\ldots,-1,n-1)$ and $\mathbf{y}=(1,\ldots,1)$, then the zeros occur at $t=\pm 1$, and the pole occurs at $t=0$. When $m=2$, the set of zeros is the union $x_1x_2=0$ of the coordinate axes, and the set of poles is the line $x_1+x_2=0$.
\end{eg}
\section{Characteristic polynomials and graph colorings}
\label{sec:char}

We prove Theorem \ref{thm:cpintro} after discussing the precoloring polynomial. This polynomial has been studied in connection with various combinatorial objects and games \cite{herzberg2007, lutz2019hyp, stanley2015}. It is also the fundamental object of the precoloring extension problem \cite{biro1992}. We assume that $D$ is simple; the proofs are essentially unchanged.

\subsection{Results from hyperplane arrangements}

Given a matroid or hyperplane arrangement $M$, write $\cp{M}$ for the characteristic polynomial of $M$. If $M$ is a matroid, then write $\rcp{M}$ for the \emph{reduced characteristic polynomial} of $M$, given by
\begin{equation*}
	\rcp{M}(\l) = (\l-1)^{-1}\cp{M}(\l).
\end{equation*}
Assume that $\g$ is loopless, and write $\cp{\g}$ for the chromatic polynomial of $\g$. We have
\begin{equation}
	\cp{\g}(\l)= \l \cp{M(\g)}(\l).
	\label{eq:chromdef}
\end{equation}

\begin{mydef}
The \emph{precoloring polynomial} of $D$ is the reduced characteristic polynomial of $M(D)$:
\begin{equation}
	\cp{D} = \rcp{M(D)}
\end{equation}
\end{mydef}

If $k\geq |\B|$ is an integer, then $\cp{D}(k)$ is the number of ways to extend a given injective $k$-coloring of $\B$ to a proper $k$-coloring of $\g$.

\begin{prop}[{\cite[Proposition 3.9]{lutz2019hyp}}]
We have
\begin{equation}
	\cp{D}(\l)=(\l)_m^{-1}\cp{\mg}(\l),
	\label{eq:cpformula}
\end{equation}
where $(\l)_m=\l(\l-1)\cdots(\l-m+1)$ is a falling factorial.
\end{prop}

\begin{eg}
Suppose that $D=D_{m,n}$ as in Example \ref{eg:dmn}. Here $\mg = K_{m+n}$ is complete, so that $\cp{\mg}(\l)=(\l)_{m+n}$. The formula \eqref{eq:cpformula} gives
\begin{equation*}
	\cp{D}(\l) = (\l-m)_n.
\end{equation*}
\end{eg}

\begin{eg}
Let $D$ be the network whose interior vertices form a cycle and whose boundary nodes are pendants, with each interior vertex adjacent to exactly 1 boundary node. The case $m=6$ is illustrated in Figure \ref{fig:hexnet}. We propose the following closed form and recurrence relation for $\cp{D}$, which we have verified for $m\leq 11$. Write $\cp{m} = \cp{D}$.
	
\begin{conj}
For $m\geq 3$ we have
\begin{equation}
	\cp{m}(\l+1) 
		= \frac{\omega_+^m + \omega_-^m}{2^m} + (-1)^m(\l-m-1),
\end{equation}
where $\omega_\pm = \l-2\pm\sqrt{\l^2+4}$. In particular, $\cp{m}$ satisfies the recurrence
\begin{equation}
	\begin{aligned}
		\cp{3}(\l)&=\l^3 - 6\l^2 + 14\l - 13\\
		\cp{4}(\l)&=\l^4 - 8\l^3 + 28\l^2 -  51\l + 41\\
		\cp{m+2}(\l)&= (\l-1)\cp{m+1}(\l)+(\l+1)\cp{m}(\l)+(-1)^m(-2\l+m-1).
	\end{aligned}
\end{equation}
\end{conj}

\begin{figure}[ht]
	\centering
	\begin{tikzpicture}[scale=1.6]
\def\a{3}
\def\b{1}
\def\c{0.5}

\foreach \s in {1,2,3,4,5,6}
{
	\coordinate (a\s) at (60*\s:1);
	\coordinate (b\s) at (60*\s:\c);
	\draw (a\s) -- (b\s);
	\draw[fill=black] (a\s) circle (1.5pt);
	\draw[fill=white] (b\s) circle (1.5pt);
};
\draw (a1) -- (a2) -- (a3) -- (a4) -- (a5) -- (a6) -- (a1);
\end{tikzpicture}
	\caption{A network with boundary nodes marked in white.}
	\label{fig:hexnet}
\end{figure}

\label{eg:hexwheel}
\end{eg}

\subsection{Proof of Theorem \ref{thm:cpintro}}
\label{sec:bcc}

Let $M$ be a matroid on $E$. With respect to a given ordering of $E$, a \emph{broken circuit} of $M$ is a set $C\setminus \min (C)$, where $C$ is a circuit of $M$. The \emph{broken circuit complex} of $M$ is the abstract simplicial complex
\begin{equation}
	\bc(M)=\{S\subseteq E : S\mbox{ contains no broken circuit of }M\}.
\end{equation}
The $f$-vector of $\bc(M)$ does not depend on the ordering of $E$:

\begin{prop}[{\cite[Theorem 4.12]{stanley2007}}]
With respect to any ordering of $E$, the number of $i$-element sets in $\bc(M)$ is $(-1)^i$ times the coefficient of $\l^{\rk(M)-i}$ in $\cp{M}(\l)$.
\label{prop:broken}
\end{prop}

We also consider the \emph{reduced broken circuit complex} $\rbc(M)$, obtained from $\bc(M)$ by deleting the minimal element of $E$ and all faces containing it. Every facet of $\bc(M)$ contains the minimal element of $E$, so $\bc(M)$ is easily recovered from $\rbc(M)$.

\begin{cor}
With respect to any ordering of $E$, the number of $i$-element sets in $\rbc(M)$ is $(-1)^i$ times the coefficient of $\l^{\rk(M)-i-1}$ in $\rcp{M}(\l)$.
\label{cor:rbroken}
\end{cor}

Recall the description of circuits of $M(D)$ from Proposition \ref{prop:circuits}. Note that the circuits of type (C) come in two flavors: one contains 3 distinct crossings, while the other contains only 2. These are illustrated in Figure \ref{fig:circuits}. Circuits of type (C) containing only 2 distinct crossings are either disconnected, as pictured, or connected with both crossings meeting at a single boundary node.

\begin{figure}[ht]
	\centering
		\begin{tikzpicture}
	\def\b{25}
	\coordinate (o) at (0,0);
	\coordinate (a1) at (180-\b:1);
	\coordinate (a3) at (180-\b:2);
	\coordinate (b1) at (180+\b:1);
	\coordinate (b2) at (180+\b:2);
	\coordinate (c1) at (0:1);
	\coordinate (c2) at (0:2);
	\coordinate (c3) at (0:3);
	\draw (o) -- (a3);
	\draw (o) -- (b2);
	\draw (o) -- (c3);
	\draw[fill=black] (o) circle (2.5pt);
	\draw[fill=black] (a1) circle (2.5pt);
	\draw[fill=white] (a3) circle (2.5pt);
	\draw[fill=white] (b2) circle (2.5pt);
	\draw[fill=black] (b1) circle (2.5pt);
	\draw[fill=black] (c2) circle (2.5pt);
	\draw[fill=black] (c1) circle (2.5pt);
	\draw[fill=white] (c3) circle (2.5pt);
	
	\def\a{7}
	\coordinate (f1) at (\a-2,0.8);
	\coordinate (f2) at (\a-3/4,0.4);
	\coordinate (f3) at (\a+3/4,0.4);
	\coordinate (f4) at (\a+2,0.8);
	\coordinate (g1) at (\a-2,-0.8);
	\coordinate (g2) at (\a-1.1,-0.45);
	\coordinate (g3) at (\a,-0.3);
	\coordinate (g4) at (\a+1.1,-0.45);
	\coordinate (g5) at (\a+2,-0.8);
	\draw (f1) -- (f2) -- (f3) -- (f4);
	\draw (g1) -- (g2) -- (g3) -- (g4) -- (g5);
	\draw[fill=white] (f1) circle (2.5pt);
	\draw[fill=black] (f2) circle (2.5pt);
	\draw[fill=black] (f3) circle (2.5pt);
	\draw[fill=white] (f4) circle (2.5pt);
	\draw[fill=white] (g1) circle (2.5pt);
	\draw[fill=black] (g2) circle (2.5pt);
	\draw[fill=black] (g3) circle (2.5pt);
	\draw[fill=black] (g4) circle (2.5pt);
	\draw[fill=white] (g5) circle (2.5pt);
	\end{tikzpicture}
	\caption{Two circuits of type (C) in Proposition \ref{prop:circuits}.}
	\label{fig:circuits}
\end{figure}

\begin{lem}
Fix an ordering of $E$, and extend this ordering to $\eo$ by taking $\eh$ to be minimal. The reduced broken circuit complex $\rbc(M(D))$ is a subcomplex of $\bc(M(\g))$:
\begin{equation}
	\rbc(M(D)) \subseteq \bc(M(\g)).
\end{equation}
	
\begin{proof}
Let $S\in \rbc(M(D))$, so that $S\subseteq E$ contains no broken circuit of $M(D)$. A cycle of $\g$ is a circuit of $M(D)$ of type (B) if it meets at most 1 boundary node, or type (C) if it meets exactly 2 boundary nodes. If a cycle $C$ of $\g$ meets 3 or more boundary nodes, then every element of $C$ is contained in a circuit $Y\subseteq C$ of type (C). Any broken circuit of $M(\g)$ is a cycle of $\g$ minus its minimal element. Thus $S$ contains no broken circuit of $M(\g)$. A broken circuit of $M(D)$ arising from a type (B) circuit is a crossing. Hence $S$ contains no crossing.
	
Now suppose instead that $S\in \bc(M(\g))$ contains no crossing. Since $S$ contains no broken circuit of $M(\g)$, it contains no broken circuit of $M(D)$ arising from a type (B) circuit. Since $S$ contains no crossing, it contains no broken circuit of $M(D)$ arising from a circuit of type (A) or (C). Hence $X$ contains no broken circuit of $M(D)$.
\end{proof}
\label{lem:rbc}
\end{lem}
	
\begin{proof}[{Proof of Theorem \ref{thm:cpintro}}]
This follows from Proposition \ref{prop:broken}, Corollary \ref{cor:rbroken} and Lemma \ref{lem:rbc}.
\end{proof}

\begin{eg}
If $\g$ is a path graph and $\B$ consists of the terminal vertices, then the only path between boundary vertices has $n+1$ edges. The last part of Theorem \ref{thm:cpintro} says that we can simply read off the coefficients $a_i$ of $\cp{\g}$ as the coefficients $b_i$ of $\cp{D}$.
\end{eg}

\appendix
\section{Dirichlet matroids and biased graphs}
\label{sec:appa}

A biased graph is a graph with a distinguished set of cycles satisfying a certain linearity condition. There are several ways to form a matroid from the distinguished cycles. We will focus on one of these (the \emph{complete lift matroid}) and apply it to a certain class of biased graphs (the \emph{almost balanced} ones). We assume that $D$ is simple to reduce headache.

\subsection{Background}

A \emph{theta graph} is a graph consisting of 2 ``terminal'' vertices and 3 internally vertex-disjoint paths between the terminals. In other words, a theta graph resembles the symbol $\theta$ (see Figure \ref{fig:thetagraph}). A set $\BB$ of cycles of $\g$ is a \emph{linear subclass} of $\g$ if, for any 2 distinct cycles in $\BB$ belonging to a theta subgraph $H$ of $\g$, the third cycle of $H$ also belongs to $\BB$.

\begin{figure}[ht]
	\centering
	\begin{tikzpicture}[scale=1.8]
\coordinate (a) at (0,0);
\coordinate (b) at (2,0);
\coordinate (c1) at (0.5,-0.1);
\coordinate (c2) at (1,0.033);
\coordinate (c3) at (1.5,-0.033);
\coordinate (d1) at (1/3,0.37);
\coordinate (d2) at (1,0.56);
\coordinate (d3) at (5/3,0.44);
\coordinate (e1) at (4/15,-0.4);
\coordinate (e2) at (11.6/15,-0.6);
\coordinate (e3) at (18.4/15,-0.6);
\coordinate (e4) at (26/15,-0.4);
\draw (a) -- (c1) -- (c2) -- (c3) -- (b);
\draw (a) -- (d1) -- (d2) -- (d3) -- (b);
\draw (a) -- (e1) -- (e2) -- (e3) -- (e4) -- (b);
\foreach \s in {a,b,c1,c2,c3,d1,d2,d3,e1,e2,e3,e4}
{ \draw[fill=black] (\s) circle (1.2pt);};
\end{tikzpicture}
	\caption{A theta graph.}
	\label{fig:thetagraph}
\end{figure}

A \emph{biased graph} is a pair $\bg=(\g,\BB)$ where $\BB$ is a linear subclass of $\g$. The cycles in $\BB$ are called \emph{balanced}; all others are \emph{unbalanced}. An edge set or subgraph $X$ is \emph{balanced} if every cycle of $\g$ contained in $X$ is balanced; otherwise $X$ is \emph{unbalanced}.

There are three matroids typically associated to a biased graph $\bg$ \cite{zaslavsky1991}. We are interested in the \emph{complete lift matroid} $L_0(\bg)$ of $\bg$. This is the matroid on $\eo = E\cup \eh$ with rank function given by
\begin{equation}
	\rk_{L_0(\bg)}(X) = 
	\begin{cases}
		|V| - c(X) & \mbox{if } X \subseteq E \mbox{ is balanced}\\
		|V| - c(X) + 1 & \mbox{if } X \subseteq E \mbox{ is unbalanced or } \eh \in X,
	\end{cases}
\end{equation}
where $c(X)$ is the number of components of the graph $(V,X)$.

\subsection{Bias representation}
\label{sec:bias}

For any $S\subseteq E$, let $\bg/S=(\g/S, \BB/S)$, where $\g/S$ is obtained from $\g$ by contracting $S$ and
\begin{multline}
	\BB/S = \{C \in \BB : C \subseteq E \setminus S
		\mbox{ and } C \mbox{ is a cycle of } \g/S\}\\
	\cup \{C \setminus S : C\in \BB \mbox{ and }
		C \cap S \mbox{ is a simple path} \}.
\end{multline}

Let $\mg$ be the graph obtained from $\g$ by adding an edge between each pair of boundary nodes, and let $\me$ be the set of added edges. We associate to $D$ the biased graph
\begin{equation*}
	\bg(D) = \bg(\mg)/\me
\end{equation*}
and denote the underlying graph by $\og$. A cycle $C$ of $\bg(D)$ is unbalanced if and only if $C$ is a crossing of $D$. We write
\begin{equation*}
	L_0(D)=L_0(\bg(D)).
\end{equation*}

\begin{eg}
Recall the network $D$ from Example \ref{eg:hexwheel}. The case $|\B|=6$ is illustrated on the left-hand side of Figure \ref{fig:hexwheel}. In this example there is only one balanced cycle of $\bg(D)$, and it is the unique cycle of $\g$.
	
\begin{figure}[ht]
	\centering
	\begin{tikzpicture}[scale=1.6]
\def\a{2.7}
\def\b{1}
\def\c{0.5}

\coordinate (d) at (2*\a,0);
\draw[fill=black] (d) circle (1.5pt);

\foreach \s in {1,2,3,4,5,6}
{
	\coordinate (a\s) at (60*\s:1);
	\coordinate (b\s) at (60*\s:\c);
	\draw (a\s) -- (b\s);
	\draw[fill=black] (a\s) circle (1.5pt);
	\draw[fill=white] (b\s) circle (1.5pt);
	\coordinate (c\s) at ($(2*\a,0)+ \b*(a\s)$);
	\draw[fill=black] (c\s) circle (1.5pt);
	\draw (c\s) -- (d);
	\coordinate (e\s) at ($(\a,0)+ \b*(a\s)$);
	\coordinate (f\s) at ($(\a,0)+ \b*(b\s)$);
	\draw[fill=black] (e\s) circle (1.5pt);
	\draw[fill=black] (f\s) circle (1.5pt);
};
\draw (a1) -- (a2) -- (a3) -- (a4) -- (a5) -- (a6) -- (a1);
\draw (c1) -- (c2) -- (c3) -- (c4) -- (c5) -- (c6) -- (c1);

\draw (e1) -- (e2) -- (e3) -- (e4) -- (e5) -- (e6) -- (e1);
\draw (f1) -- (f2) -- (f3) -- (f4) -- (f5) -- (f6) -- (f1);
\draw (e1) -- (f1);
\draw (e2) -- (f2);
\draw (e3) -- (f3);
\draw (e4) -- (f4);
\draw (e5) -- (f5);
\draw (e6) -- (f6);
\draw (f1) -- (f3);
\draw (f1) -- (f4);
\draw (f1) -- (f5);
\draw (f2) -- (f4);
\draw (f2) -- (f5);
\draw (f2) -- (f6);
\draw (f3) -- (f5);
\draw (f3) -- (f6);
\draw (f4) -- (f6);
\end{tikzpicture}
	\caption{Left to right: a network $D$ with boundary nodes marked
		in white, the graph $\mg$, and the graph $\og$.}
	\label{fig:hexwheel}
\end{figure}
\end{eg}

There is a characterization of the biased graphs $\bg(D)$ due to \cite{zaslavsky1987}. For $i\in V$ let $\bg\setminus i$ be the biased graph obtained by deleting $i$ and all edges incident to $i$. If $\bg$ is unbalanced but $\bg\setminus i$ is balanced, then $i$ is called a \emph{balancing vertex} of $\bg$. A biased graph with a unique balancing vertex is called \emph{almost balanced}. We now have all the terms needed to characterize Dirichlet matroids in terms of biased graphs.

\begin{thm}
The classes of simple Dirichlet matroids and simple complete lift matroids of connected almost balanced biased graphs are equal. In particular, we have $M(D)=L_0(D)$.

\begin{proof}
This follows from \cite[Proposition 1]{zaslavsky1987} and Proposition \ref{prop:samemat} below.
\end{proof}
\end{thm}

Propositions \ref{prop:circuits}, \ref{prop:bases} and \ref{prop:cocircuits} now follow from parts (e), (g) and (h) of \cite[Theorem 3.1]{zaslavsky1991}, respectively.

\subsection{Gain representation}

A \emph{gain graph} $\Phi=(\g,\gr,\phi)$ consists of
\begin{enumerate}[(i)]
	\item A graph $\g$
	\item A group $\gr$
	\item A function $\phi:V\times V\to \gr$ such that $\phi(i,j)=\phi(j,i)^{-1}$ for all $(i,j)$.
\end{enumerate}
If $ij\in E$, then we think of $(i,j)$ as the edge $ij$ oriented from $i$ to $j$.
	
For any cycle $C$ of $\g$, order the vertices of $C$ in a cycle as $i_1,\ldots,i_\ell=i_1$, and write \begin{equation}\phi(C)=\phi(i_1,i_2)\phi(i_2,i_3)\cdots\phi(i_{\ell-1},i_\ell).\end{equation}
In general, the element $\phi(C)$ depends on the choice of starting vertex and direction, unless $\phi(C)$ is the identity. Let
\begin{equation}
	\BB= \{ C \subseteq E: C \mbox{ is a cycle of } \g
	\mbox{ with } \phi(C) \mbox{ the identity of } \gr \}.
\end{equation}
The set $\BB$ is a linear subclass of $\g$. Thus every gain graph defines a biased graph whose set of balanced cycles is $\BB$.

Let $\KK$ be the additive group of a field. For any $u\in \KK^{\B}$, let $\Phi(\g,u)=(\og,\KK,\phi)$, where $\og$ is defined as above and $\phi:V\times V\to \KK$ is given by
\begin{equation}
	\phi(i,j) = 
	\begin{cases}
		u_j & \mbox{if } ij \in \be \mbox{ with } j \in \B\\
		-u_i & \mbox{if } ij \in \be \mbox{ with } i\in \B\\
		0 & \mbox{else}.
	\end{cases}
\end{equation}
Clearly $\Phi(\g,u)$ is a gain graph.

\begin{eg}
Consider the graph $\g$ on the left side of Figure \ref{fig:gg} with boundary nodes marked in white and values of $u$ labeled. The associated gain graph $\Phi(\g,u)$ is illustrated on the right side of Figure \ref{fig:gg}. An edge oriented from $i$ to $j$ with label $k$ means that $\phi(i,j)=k$.
	
\begin{figure}[ht]
	\centering
	\begin{tikzpicture}[scale=1.8]
\tikzset{->-/.style={decoration={
			markings,
			mark=at position .5 with {\arrow{>}}},postaction={decorate}}}
	\def\c{0.75}
	\coordinate (a) at ({-2*\c},0);
	\coordinate (b) at ({-3*\c},0);
	\coordinate (h) at (0,0);
	\draw[fill=black] (a) circle (1.45pt);
	\draw[fill=black] (b) circle (1.45pt);
	\draw[fill=black] (h) circle (1.45pt);
	\draw[->-] (b) -- (a);
	\draw[->-] (a) .. controls ({-1.33*\c},0.2) and ({-0.66*\c},0.2) .. (h);
	\draw[->-] (a) .. controls ({-1.33*\c},-0.2) and ({-0.66*\c},-0.2) .. (h);
	\draw[->-] (b) .. controls ({-2.7*\c},1) and  ({-0.4*\c},1) .. (h);
	\draw[->-] (b) .. controls ({-2.7*\c},-1) and ({-0.4*\c},-1) .. (h);
	\node at ({-2.5*\c},0) [below = 1.1mm] {$1$};
	\node at (-0.1,-0.5) {$a$};
	\node at (-0.08,0.47) {$b$};
	\node at (-1.08,0.32) {$b$};
	\node at (-1.1,-0.35) {$a$};
	

	\def\a{-4.5}
	\def\b{0.47}
	\coordinate (A) at ({\a-\b*2.2},0);
	\coordinate (B) at (\a,{\b*1.45});
	\coordinate (C) at (\a,{-1.45*\b});
	\coordinate (D) at ({\a+\b*2.2},0);
	\draw (A) -- (B) -- (D) -- (C) -- (A);
	\draw (B) -- (C);
	\node at (A) [below = 1.2mm] {$a$};
	\node at (D) [below = 1.2mm] {$b$};
	\draw[fill=white] (A) circle (1.45pt);
	\draw[fill=black] (B) circle (1.45pt);
	\draw[fill=black] (C) circle (1.45pt);
	\draw[fill=white] (D) circle (1.45pt);
	
\end{tikzpicture}
	\caption{A network and the associated gain graph.}
	\label{fig:gg}
\end{figure}
\end{eg}
	
\subsection{Block injectivity and equivalence of representations}

Recall the notion of a \emph{block} of $D$ from Definition \ref{def:block}.

\begin{mydef}
We say that $u\in \KK^{\B}$ is \emph{block injective} if for every block $U$ of $D$ the values $u_i$ are distinct for all $i\in U\cap \B$.
\end{mydef}

A cycle $C$ of the gain graph $\Phi(\g,u)$ is unbalanced if and only if $C$ is a crossing of $D$ between boundary nodes on which $u$ takes distinct values, so $\Phi(\g,u)$ is independent of $u$, as long as $u$ is block injective. Write
\begin{equation*}
	\Phi(D)=\Phi(\g,u),
\end{equation*}
where $u$ is block injective. The following proposition justifies our claim in Example \ref{eg:dirarr} that the matroids underlying Dirichlet arrangements are the simple Dirichlet matroids.

\begin{prop}
The following matroids on $\eo$ are equal:
\begin{enumerate}[(a)]
	\item $M(D)$
	\item $L_0(\bg(D))$
	\item $L_0(\Phi(D))$
\end{enumerate}
If $D$ is simple and $u\in \KK^{\B}$ is block injective, then we can add:
\begin{enumerate}[(a)]
	\setcounter{enumi}{3}
	\item The matroid defined by the Dirichlet arrangement $\AA_D(u)$.
\end{enumerate}
	
\begin{proof}
A cycle of $\bg(D)$ or $\Phi(D)$ is unbalanced if and only if it is a crossing of $D$. Hence $\Phi(D)=\bg(D)$ as biased graphs, proving the equivalence of (b) and (c). The equivalence of (c) and (d) follows from \cite[Theorem 4.1(a)]{zaslavsky2003}. In addition it follows that $\AA(\g,u)$ defines the same matroid for any field $\KK$ with a block-injective boundary data $u$, proving the equivalence of (a) and (d).
\end{proof}
\label{prop:samemat}
\end{prop}

\subsection{Proof of Proposition \ref{prop:representable}}
\label{sec:reppf}

We need a description of the single-element contractions of $M(D)$.

\begin{lem}
If $e\in E$, then $M(D)/e = L_0(\bg(D)/e)$, where $\bg(D)/e$ is the biased graph with underlying graph $\og/e$ and in which a cycle $C\subseteq E\setminus e$ of $\og/e$ is balanced if and only if $C\cup e$ is a balanced cycle of $\bg(D)$.
	
\begin{proof}
The result follows from the discussion in \cite[p. 38]{zaslavsky1989}.
\end{proof}
\label{lem:contract}
\end{lem}

\begin{proof}[Proof of Proposition \ref{prop:representable}]
Suppose that $|\KK|\geq \omega(D)$, so that there exists a block-injective $u\in \KK^{\B}$. Proposition \ref{prop:samemat} implies that $M(D)$ is representable over $\KK$, since any hyperplane representation over $\KK$ gives a matrix representation over $\KK$.
		
Now suppose that $|\KK|<\omega(D)$, and let $U$ be a block with $\omega(D) = |U\cap \B|$. Let $X\subseteq E$ be the set of all edges with both endpoints in $U$. Let $\partial X\subseteq X$ be the set of all edges with one endpoint in $U\cap \B$. Deleting all edges in $E\setminus X$ and contracting all edges in $X\setminus\partial X$ yields the star network $S_{\omega(D)+1}$ (see Example \ref{eg:star}). Since $M(S_d)\cong U_{2,d}$, we obtain $U_{2,\omega(D)+1}$ as a minor of $M(D)$ by Lemma \ref{lem:contract}. But $U_{2,\omega(D)+1}$ is not a minor of any matroid representable over $\KK$ \cite[Corollary 6.5.3]{oxley2006}.
\end{proof}

\section{Complete principal truncations}
\label{sec:printrun}

Let $M$ be a matroid on $E$ and $F$ a flat of rank $r$. Let $M_F$ denote the \emph{complete principal truncation} of $M$ along $F$, obtained by freely adding an $(r-1)$-element independent set $I$ to $F$ and contracting $I$ (see \cite[p. 266]{oxley2006} for details). Note that $F$ is a parallel class of $M_F$; we choose a representative $e\in F$ and treat $M$ as a matroid on $(E\setminus F)\cup e$ by deleting $E\setminus e$.

Let $\mg$ denote the graph obtained from $\g$ by adding an edge between each pair of boundary vertices. Let $E(\B)$ denote the set of added edges.

\begin{thm}
The Dirichlet matroid $M(D)$ is isomorphic to the complete principal truncation of $M(\mg)$ along $E(\B)$.
	
\begin{proof}
This follows from the discussion in \cite[Section 6]{lutz2019hyp}.
\end{proof}
\end{thm}

The following result gives an alternative proof of Proposition \ref{prop:hpp}, although it bypasses the connection to the response matrix.

\begin{prop}
If a matroid $M$ has the half-plane property, then so does the complete principal truncation of $M$ along any flat.
	
\begin{proof}
This follows from Propositions 4.11 and 4.12 of \cite{choe2004}.
\end{proof}
\end{prop}

\section*{Acknowledgments}
The author thanks H\'{e}l\`{e}ne Barcelo, Trevor Hyde, Jeffrey C. Lagarias and Thomas Zaslavsky for comments on drafts of this article; Anastasia Chavez, Christopher Eur and James G. Oxley for discussions and references on matroid quotients; and the anonymous referee for suggestions and improvements.

\bibliographystyle{abbrv}
\bibliography{lutz-mafen}
\end{document}